\documentclass[11pt]{amsart}

\usepackage[colorlinks=true, pdfstartview=FitV, linkcolor=blue, 
citecolor=blue]{hyperref}

\usepackage{a4wide}
\usepackage{amsmath,amssymb} 
\usepackage{bbm}
\usepackage{graphicx}

\theoremstyle{plain}
\newtheorem{theorem}{Theorem}[section]
\newtheorem{prop}[theorem]{Proposition}
\newtheorem{lemma}[theorem]{Lemma}
\newtheorem{coro}[theorem]{Corollary}
\newtheorem{fact}[theorem]{Fact}

\theoremstyle{definition}

\newtheorem{example}[theorem]{Example}
\newtheorem{remark}[theorem]{Remark}

\newcommand{\dd}{\,\mathrm{d}}
\newcommand{\ii}{\pts\mathrm{i}\pts}
\newcommand{\ee}{\mathrm{e}}

\newcommand{\pts}{\hspace{0.5pt}}
\newcommand{\nts}{\hspace{-0.5pt}}

\newcommand{\fm}{\mathfrak{m}\pts}
\newcommand{\one}{\mathbbm{1}}

\newcommand{\ZZ}{\mathbb{Z}\pts}
\newcommand{\QQ}{\mathbb{Q}}
\newcommand{\RR}{\mathbb{R}}
\newcommand{\NN}{\mathbb{N}}
\newcommand{\CC}{\mathbb{C}}
\newcommand{\TT}{\mathbb{T}}
\newcommand{\MM}{\mathbb{M}}
\newcommand{\XX}{\mathbb{X}}
\newcommand{\YY}{\mathbb{Y}}

\newcommand{\cA}{\mathcal{A}}
\newcommand{\cE}{\mathcal{E}}
\newcommand{\cT}{\mathcal{T}}
\newcommand{\ct}{\mathfrak{t}}

\newcommand{\vG}{\varGamma}
\newcommand{\vL}{\varLambda}
\newcommand{\vU}{\varUpsilon}

\newcommand{\pa}{\phantom{a}}
\newcommand{\exend}{\hfill $\Diamond$}

\DeclareMathOperator{\dens}{dens}
\DeclareMathOperator{\diag}{diag}
\DeclareMathOperator{\card}{card}
\DeclareMathOperator{\vol}{vol}
\DeclareMathOperator{\Mat}{Mat}
\DeclareMathOperator*{\Conv}{\text{\Huge \raisebox{-3.5pt}{$\ast$}}}

\newcommand{\myfrac}[2]{\frac{\raisebox{-2pt}{$#1$}}
      {\raisebox{0.5pt}{$#2$}}}

\begin{document}

\title[Renormalisation for block substitutions]
{Renormalisation of pair correlations and their\\[2mm]
 Fourier transforms for primitive block substitutions}

\author{Michael Baake}
\address{Fakult\"at f\"ur Mathematik, Universit\"at Bielefeld, \newline
\hspace*{\parindent}Postfach 100131, 33501 Bielefeld, Germany}
\email{mbaake@math.uni-bielefeld.de }

\author{Uwe Grimm}
\address{School of Mathematics and Statistics,
  The Open University,\newline \hspace*{\parindent}Walton Hall, 
  Milton Keynes MK7 6AA, United Kingdom} 
\email{uwe.grimm@open.ac.uk}

\begin{abstract}  
  For point sets and tilings that can be constructed with the
  projection method, one has a good understanding of the correlation
  structure, and also of the corresponding spectra, both in the
  dynamical and in the diffraction sense. For systems defined by
  substitution or inflation rules, the situation is less favourable,
  in particular beyond the much-studied class of Pisot
  substitutions. In this contribution, the geometric inflation rule is
  employed to access the pair correlation measures of self-similar and
  self-affine inflation tilings and their Fourier transforms by means
  of exact renormalisation relations.  In particular, we look into
  sufficient criteria for the absence of absolutely continuous
  spectral contributions, and illustrate this with examples from the
  class of block substitutions. We also discuss the Frank--Robinson
  tiling, as a planar example with infinite local complexity and
  singular continuous spectrum.
\end{abstract}

\maketitle

\section{Introduction}

The theory of model sets via the projection method, see
\cite{TAO,Nicu} and references therein for background, has led to a
reasonably good understanding of mathematical models for perfect
quasicrystals. This is particularly true of systems with pure point
spectrum, and applies to spectra both in the diffraction and in the
dynamical sense; see \cite{TAO,KLS,LMS,BLvE,BL} and references therein
for more, in particular on equivalence results for the different types
of spectra.

Another intensely-studied approach starts from a substitution on a
finite alphabet, or considers an inflation rule for a finite set of
prototiles; see \cite{TAO,Nat-review,BorisNew} and <
references therein for more. If the inflation multiplier happens to be
a Pisot--Vijayaraghavan (PV) number, one meets an interesting overlap
with the projection method via systems that can both be described by
inflation and as a regular model set; see \cite[Ch.~7]{TAO} for some
classic examples. However, the still open Pisot substitution
conjecture, compare \cite{Aki,Bernd}, shows that important parts of
the picture are still missing.

Considerably less is known for more general substitution or inflation
schemes, be it beyond the PV case, in higher dimensions, or both. In
particular, the study of non-PV substitutions is only at its
beginning.  Some recent progress \cite{BFGR,BGM} in one dimension was
possible by realising that such systems admit an exact renormalisation
approach to their pair correlation measures; see \cite{BuSol,BuSol2}
for related results on the spectral measures for these systems.

The purpose of this contribution is to show how to extend such an
exact renormalisation approach to higher dimensions, and also beyond
the case of inflation tilings of finite local complexity (FLC).  To be
able to discuss some interesting classes of examples, we will build on
several results from \cite{BGM}.  One of our goals is to formulate an
effective sufficient criterion for the absence of absolutely
continuous (\textsf{ac}) diffraction, which then implies that the
diffraction measure is a singular measure, with the analogous result
on the spectral measure of maximal type where possible at present.
This clearly is expected to be the typical situation for inflation
systems with vanishing topological entropy, but no general
classification is known so far.

To formulate a criterion for the absence of \textsf{ac} components, it
will be instrumental to identify a natural cocycle attached to the
inflation rule together with an appropriate Lyapunov
exponent. Implicitly, this amounts to an asymptotic analysis of
infinite matrix products of Riesz product type. They have shown up in
various ways in the spectral theory of inflation systems
\cite{Q,BuSol,BFGR,BaGriMa,BuSol2}. It should not be surprising to
meet them again, in a slightly different fashion. In fact, they
provide perhaps the most natural point of entry for a renormalisation
type analysis of inflation systems.

This contribution is both a summary of known results, including those
from \cite{BG15,Neil,BGM,BCM}, and their extension to some new
territory, in particular in higher dimensions (as announced in
\cite{Neil-MFO} and discussed in \cite{BGM}). For the latter purpose,
we proceed in an example-oriented manner via the class of block
substitutions (not necessarily of constant size), which is still
sufficiently simple to see the underlying ideas, yet rich enough to
illustrate some new phenomena. In particular, in view of recent
general interest \cite{FS1,FS2,Nat-ILC,Dirk,LS,BorisNew}, we include
some examples of infinite local complexity as well.  Various general
results that we employ are discussed and proved in \cite{BGM}, for
which we only give a brief account here.\smallskip

The material presented below is organised as follows. In
Section~\ref{sec:prelim}, we set the scene by recalling some basic
material, including some proofs for convenience, in particular where
we are not aware of a good reference. Section~\ref{sec:sample}
continues this account, covering some important aspects of uniform
distribution and averages, which will be instrumental in most of our
later calculations.  Then, in Section~\ref{sec:1d}, we discuss
inflation systems in one dimension, from the viewpoint of exact
renormalisation of the pair correlation measures and their Fourier
transforms, with one concrete example of recent interest being
discussed in Section~\ref{sec:ex-no-ac}.  For further fully worked-out
examples, we refer to \cite{Neil,BG15,BFGR,BaGriMa}.

Starting with Section~\ref{sec:general}, we develop the entire theory
for higher-dimensional inflation tilings with finitely many prototiles
up to translations, which is then applied to various examples. In
particular, we treat binary block substitutions of constant size
(Section~\ref{sec:constant}) and a rather versatile family of block
substitutions with squares (Section~\ref{sec:block}), which comprises
tilings with infinite local complexity. This is also a feature of the
Frank--Robinson tiling (Section~\ref{sec:FR}), which is shown to have
singular continuous diffraction beyond the trivial Bragg peak at the
origin. Some concluding remarks and open problems follow in
Section~\ref{sec:outlook}.

\section{Preliminaries}\label{sec:prelim}

Our general references for concepts, notation and background are
\cite{TAO,TAO2}.  Here, we collect further methods and results, where
we begin with a simple property of Hermitian matrices.

\begin{fact}\label{fact:decompose}
  Let\/ $H = (h_{ij})^{\pa}_{1\leqslant i,j \leqslant d}\in \Mat (d,\CC)$
  be Hermitian and positive semi-definite, with rank\/ $m$.  Then, all
  diagonal elements of\/ $H$ are non-negative. If\/ $h_{ii} =0$ for
  some\/ $i$, one has\/ $h_{ij} = h_{ji} = 0$ for all\/
  $1 \leqslant j \leqslant d$.  In particular, $H=0$ iff\/ $m=0$.
  
  Whenever\/ $H\ne 0$, there are\/ $m$ Hermitian, positive
  semi-definite matrices\/ $H_1, \ldots , H_m$ of rank\/ $1$ such
  that\/ $H = \sum_{r=1}^{m} H_r$ together with\/ $H_r \pts H_s = 0$
  for\/ $r\ne s$.
\end{fact}

\begin{proof}
  By Sylvester's criterion, $H$ positive semi-definite means that all
  principal minors are non-negative, hence in particular all diagonal
  elements of $H$. Assume $h_{ii}=0$ for some $i$, and select any
  $j\in \{ 1, \ldots, d\}$.  By semi-definiteness in conjunction with
  Hermiticity, one finds
\[
     0 \, = \, h_{ii} \pts h_{jj} \, \geqslant \, h_{ij} \pts h_{ji}
     \, = \, \lvert h_{ij} \rvert^2  \, \geqslant \, 0 \pts ,
\]  
 which implies the second claim. The equivalence of $H=0$ with
 $m=0$ is clear.
  
 Employing Dirac's notation, the spectral theorem for Hermitian
 matrices asserts that one has
 $H = \sum_{i=1}^{d} | v_i \rangle \pts \lambda_i \pts \langle v_i |$,
 where the eigenvectors $|v_i \rangle$ can be chosen to form an
 orthonormal basis (so $\langle v_i | v_j \rangle = \delta_{i,j}$ and
 $|v_i\rangle \langle v_i |$ is a projector of rank $1$), while all
 eigenvalues are non-negative due to positive semi-definiteness. The
 rank of $H$ is the number of positive eigenvalues, counted with
 multiplicities. Ordering the eigenvalues decreasingly as
 $\lambda^{\pa}_{1} \geqslant \lambda^{\pa}_{2} \geqslant \cdots
 \geqslant \lambda^{\pa}_{d} \geqslant 0$,
 one can choose $H_r = | v_r \rangle \pts \lambda_r \pts \langle v_r|$
 for $1 \leqslant r \leqslant m$, and the claim is obvious.
\end{proof}

\subsection{Logarithmic integrals and Mahler measures}

The logarithmic Mahler measure of a polynomial $p \in \CC [x]$
is defined as
\begin{equation}\label{eq:def-Mahler}
    \fm (p) \, := \int_{0}^{1} \log \big|
    p \bigl( \ee^{2 \pi \ii t} \bigr) \big| \dd t \pts .
\end{equation}
It was originally introduced by Mahler as a measure of the complexity
of $p$; compare \cite{EvWa}.  If $p(x) = a \prod_{i=1}^{s} (x -
\alpha_i)$, it follows from Jensen's formula \cite[Prop.~16.1]{Klaus}
that
\begin{equation}\label{eq:Jensen}
    \fm (p) \, = \, \log\pts \lvert a \rvert \, +
    \sum_{i=1}^{s} \log \bigl( \max \{ 1, 
    \lvert \alpha_i \rvert \} \bigr) .
\end{equation}
This has the following immediate consequence.

\begin{fact}\label{fact:cyclo}
  If\/ $p$ is a monic polynomial that has no roots outside the unit
  disk, one has\/ $\fm (p) = 0$.  In particular, this holds if\/ $p$
  is a cyclotomic polynomial,\footnote{A non-constant polynomial $p
    \in \ZZ [x]$ is called \emph{cyclotomic} if $p (x)$ divides $x^n
    \! - \nts 1$ for some $n\in\NN$.} or a product of such a
  polynomial with a monomial.  \qed
\end{fact}

Clearly, for polynomials $p$ and $q$, one has
$\fm (p \pts q) = \fm (p) + \fm (q)$. If $p \in \ZZ [x]$, one can say
more about the possible values of $\fm (p)$. They are of interest both
in number theory and in dynamical systems; see \cite{EvWa,BCM} and
references therein.

Mahler measures of multivariate (or multi-variable) polynomials are
defined by an integration over the corresponding torus. Concretely,
for $p\in\CC \bigl[ x^{\pa}_{1}, \ldots , x^{\pa}_{d} \bigr]$, one has
\begin{equation}\label{eq:def-gen-Mahler}
   \fm (p) \, := \int_{\TT^d}
   \log \big| p \bigl(\ee^{2 \pi \ii t_1}, \ldots , 
   \ee^{2 \pi \ii t_d} \bigr) \big| \dd t^{\pa}_1  
   \cdots \dd t^{\pa}_d \pts ,
\end{equation}
where $\TT^d = \RR^d/\ZZ^d$ denotes the $d$-torus.  Unfortunately, in
contrast to the one-dimensional situation, there is no simple general
way to calculate such integrals.  If we need to single out a variable,
we do so by a subscript. For instance, $\fm_x (1 + x + xy)$ denotes
the logarithmic Mahler measure of $1+x+xy$, viewed as a polynomial in
$x$, with $y$ being a coefficient. We refer to \cite{EvWa} for general
background and examples.

\subsection{Radon measures}

Let $\mu$ denote a (generally complex) Radon measure on $\RR^d$, which
we primarily view as a linear functional over the space
$C_{\mathsf{c}} (\RR^d)$ of compactly supported continuous functions.
The `flipped-over' version $\widetilde{\mu}$ is defined by
$\widetilde{\mu} (g) = \overline{\mu (\widetilde{g}\pts )}$, where
$\widetilde{g} (x) := \overline{g (-x)}$. A measure $\mu$ is called
\emph{positive} when $\mu (g) \geqslant 0$ for all $g\geqslant 0$, and
\emph{positive definite} when
$\mu (g * \widetilde{g} \pts ) \geqslant 0$ for all
$g\in C_{\mathsf{c}} (\RR^d)$.  Here, $g*h$ refers to the convolution
of two integrable functions, as defined by
$\bigl( g * h \bigr) (x) = \int_{\RR^d} g(x-y) \pts h (y) \dd y$. By
$\lvert \mu \rvert$, we denote the \emph{total variation measure} of
$\mu$.  If $\lvert \mu \rvert (\RR^d) < \infty$, the measure is
\emph{bounded} or \emph{finite}, while we call it
\emph{translation-bounded} when
$\sup_{t\in\RR^d} \lvert \mu \rvert (t + K) < \infty$ holds for some
compact set $K\subset \RR^d$ with non-empty interior.

If $f \! : \, \RR^d \xrightarrow{\quad} \RR^d$ is an invertible
mapping, we define the \emph{pushforward} $f \! . \mu$ of a measure
$\mu$ by $\bigl(f \! . \mu\bigr) (g) = \mu ( g \nts\circ\! f)$, where
$g$ is an arbitrary test function. Viewing $\mu$ as a regular Borel
measure via the general Riesz--Markov representation theorem, compare
\cite{Simon}, the matching relation for a bounded Borel set $\cE$ is
\[
    \bigl(f \! . \mu\bigr) (\cE) \, = \, 
    \bigl( f \! . \mu\bigr) (1^{\pa}_{\cE} ) \, = \,
    \mu ( 1^{\pa}_{\cE} \circ f ) \, = \, 
    \mu \bigl(1^{\pa}_{f^{-1} (\cE)} \bigr)
    \, = \, \mu \bigl(f^{-1} (\cE) \bigr) .
\]
Of particular importance is the \emph{Dirac measure} at $x$, denoted
by $\delta_x$, which is defined by $\delta_x (g) = g (x)$ for test
functions.  For Borel sets, the matching relation is
\[
     \delta^{\pa}_{x}  (\cE) \, = \, \begin{cases}
        1 , & \text{if $x\in\cE$} , \\ 0 , & \text{otherwise} ,
        \end{cases}
\]
which is often used in the form $\delta^{\pa}_x (\cE) = 
\delta^{\pa}_{x} \bigl( 1^{\pa}_{\cE} \bigr)$.
For a point set $S\subset \RR^d$, which is at most countable in our
setting \cite{TAO}, one defines the corresponding \emph{Dirac comb} as
$\delta^{\pa}_{S} = \sum_{x\in S} \delta^{\pa}_{x}$.

When $\nu$ is absolutely continuous relative to $\mu$, denoted by
$\nu \ll \mu$, with Radon--Nikodym density $h$, we write
$\nu = h \pts \mu$, so that
$\bigl( h \pts \mu \bigr) (g) = \mu (h \pts g)$. For the pushforward,
this leads to
\begin{equation}\label{eq:forward-1}
  f \! . (h \pts \mu) \, = \,  \bigl( h \circ \nts f^{-1} \bigr) 
  \cdot (f\! . \mu) \pts ,
\end{equation}
as follows from a simple calculation with a test function.  When
$f(x) = Ax$ with $A \in \mathrm{GL} (d,\RR)$ and $\mu$ is Lebesgue 
measure, it is sometimes more convenient to rewrite this relation as
\begin{equation}\label{eq:forward-ac}
    A.h \, := \, f \!. h \, = \,
    \frac{h \circ f^{-1}}{\lvert \det (A) \rvert}  \, = \,
    \frac{h \circ A^{-1}}{\lvert \det (A) \rvert} \pts ,
\end{equation}
to be understood as a relation between absolutely continuous measures.
The \emph{convolution} $\mu * \nu$ of two finite measures is defined by
\[
    \bigl( \mu * \nu \bigr) (g) \, = \int_{\RR^d} \int_{\RR^d}
    g (x+y) \dd \mu (x) \dd \nu (y) \pts ,
\]
which can be extended in various ways, in particular to the case where
one measure is finite and the other is translation bounded
\cite[Prop.~1.13]{BF}.

\begin{lemma}\label{lem:push-conv}
  If\/ $\mu$ and\/ $\nu$ are convolvable measures on\/ $\RR^d$ and
  if\/ $f \! : \, \RR^d \xrightarrow{\quad} \RR^d$ is invertible and
  linear, the pushforward operation satisfies\/ $\, f\! . (\mu * \nu)
  = (f\! . \mu) * (f\! . \pts \nu)$.
\end{lemma}

\begin{proof}
  Let $g$ be a general test function and define $g^{\pa}_{a} $ by
  $g^{\pa}_{a} (x) = g (a+x)$. Then, one has
\[
\begin{split}
   \bigl( f \! . (\mu * \nu) \bigr) (g)\, & = \, \int_{\RR^d}
   \int_{\RR^d}  g \bigl( f( x+y) \bigr) \dd \mu (x) \dd \nu (y) 
   \, = \, \int_{\RR^d} \int_{\RR^d}
   g \bigl( f( x ) + f(y) \bigr) \dd \mu (x) \dd \nu (y)  \\[2mm]
   & = \int_{\RR^d} \mu \bigl( g^{\pa}_{f(y)}  \circ f  \bigr) 
   \dd \nu (y) \, = \int_{\RR^d} ( f \! . \mu ) 
   \bigl( g^{\pa}_{f(y)} \bigr) \dd \nu (y) \\[2mm]
   & = \int_{\RR^d} \int_{\RR^d} g \bigl( x + f(y) \bigr)
   \dd \bigl( f \! . \mu \bigr) (x) \dd \nu (y) \, = \int_{\RR^d}
   \int_{\RR^d} g^{\pa}_{x} \bigl( f(y) \bigr) \dd \nu (y) \dd
   \bigl( f\! . \mu \bigr) (x) \\[2mm]
   & = \int_{\RR^d} \nu (g^{\pa}_{x} \circ f ) \dd 
      \bigl( f \! . \mu\bigr) (x)
   \, = \int_{\RR^d} \bigl( f\! . \pts \nu\bigr) (g^{\pa}_{x} )
    \dd \bigl( f\! . \mu \bigr) 
   (x) \\[2mm] & = \int_{\RR^d} \int_{\RR^d} g (x+y)
   \dd \bigl( f\! . \pts \nu \bigr) (y) \dd \bigl( f\! .\mu\bigr) (x) 
   \, = \, \bigl( (f\! . \pts \nu) * (f\! . \mu) \bigr) (g) \\[3mm]
   & = \,    \bigl( (f\! . \mu) * (f\! . \pts \nu) \bigr) (g) \pts ,
\end{split}
\]    
where the second step in the third line, and the last step, rely 
on Fubini's theorem.
\end{proof}

A linear map $f$ on $\RR^d$ is \emph{expansive} if there is a constant
$\alpha > 1$ such that $\| f (x) \| \geqslant \alpha \| x \|$ for all
$x \in \RR^d$. This implies that all eigenvalues satisfy
$\lvert \lambda \rvert \geqslant \alpha$ and that $f$ is invertible.

\begin{lemma}\label{lem:bei-null}
  Let\/ $\mu$ be a Radon measure on\/ $\RR^d$ such that\/
  $\mu |^{\pa}_{U} = \mu (\{ 0 \} ) \, \delta^{\pa}_{0}$ holds for
  some open neighbourhood\/ $U$ of\/ $0$. Then, if\/ $f$ is an
  expansive linear map on\/ $\RR^d$, one has
\[   
    \lim_{n\to\infty}  f^n \! . \pts \mu  \, = \, 
      \mu (\{ 0 \}) \, \delta^{\pa}_{0} \pts .
\]   
\end{lemma}

\begin{proof}
  Let $\cE\subset \RR^d$ be a fixed, bounded Borel set.  Viewing $\mu$
  as a regular Borel measure, one has
  $ \bigl( f^n\! . \mu \bigr) (\cE) = \mu \bigl( f^{-n} (\cE)\bigr)$.
  Since $f$ is expansive, with expansion constant $\alpha > 1$, it is
  invertible, and $f^{-1}$ is contractive, with
  $\| f^{-1} (x) \| \leqslant \frac{1}{\alpha} \| x \|$ for all
  $x\in\RR^d$. Consequently, $f^{-n} (\cE) \subset U$ for $n$
  sufficiently large.
  
  Now, the set $\cE$ contains $0$ if and only if $f^{-n} (\cE)$ does,
  so $\bigl( f^n \! . \mu \bigr) (\cE) = \mu(\{ 0 \}) \, \delta^{\pa}_{0}
  (\cE)$ for $n$ large enough. Since $\cE$ was bounded but otherwise
  arbitrary, our claim follows.
\end{proof}

The Fourier transform of measures will play an important role in many
of our arguments. We follow the classical approach as outlined in
\cite[Ch.~1]{BF}, see also \cite[Ch.~8]{TAO} as well as \cite{MS},
where the Fourier transform of an integrable function $f$ is given by
\[
      \widehat{f} (k) \, = \int_{\RR^d} \ee^{- 2 \pi \ii 
      \langle k | x \rangle} f(x) \dd x
\]
as usual, where $\langle \cdot | \cdot \rangle$ denotes the standard inner
product of $\RR^d$. If $\mu$ is a finite measure, its Fourier transform 
is a continuous function, written as
\[
    \widehat{\mu} (k) \, = \int_{\RR^d} \ee^{-2 \pi \ii 
    \langle k | x \rangle}  \dd \mu (x) \pts .
\]
For translation-bounded measures, we shall also employ standard notions
and techniques from the theory of tempered distributions; compare
\cite[Sec.~6.2]{Simon}.

Below, we will make frequent use of a relation that tracks the
consequence of an invertible linear map under Fourier transform.

\begin{lemma}\label{lem:linear-FT}
   Let\/ $\mu$ be a Fourier-transformable measure on\/ $\RR^d$,
   and\/ $A\in \mathrm{GL} (d,\RR)$. Then, with\/ $A^{*} := (A^T)^{-1}$
   denoting the dual matrix, one has
\[
   \widehat{A.\mu} \, = \, \frac{A^{*} \! . \pts \widehat{\mu}}
   {\lvert \det (A) \rvert} \pts .
\]
   Moreover, when\/ $\widehat{\mu}$ is absolutely continuous
   relative to Lebesgue measure, hence represented by a locally
   integrable function, the relation simplifies to
\[
    \widehat{A.\mu} \, = \, \widehat{\mu} \pts \circ
    \nts A^T \, = \, A^T \! . \pts \widehat{\mu} \pts .
\]
\end{lemma}

\begin{proof}
If $g$ is an arbitrary test function, one has
\[
   \widehat{A.\mu} \, (g) \, = \, \bigl( A.\mu \bigr) 
   (\pts\widehat{g}\pts )
   \, = \, \mu (\pts \widehat{g} \circ \nts\nts A ) \, = 
   \int_{\RR^d} \int_{\RR^d} \ee^{- 2 \pi \ii 
   \langle x | At \rangle} g(x)
   \dd x \dd \mu (t) \pts .
\]
Observing
$\langle x | A \pts t \rangle = \langle t | A^T \nts \nts x \rangle$
and setting $x = A^* y$, hence
$\dd x = \lvert \det (A^* ) \rvert \dd y$, one finds
\[
\begin{split}
   \widehat{A.\mu} \, (g) \, & = 
   \int_{\RR^d} \int_{\RR^d} \ee^{- 2 \pi \ii \langle t | y \rangle}
   g(A^* y) \, \frac{\dd y \dd \mu (t)}
   {\lvert \det (A) \rvert} \, =
   \int_{\RR^d} \bigl(\widehat{g \circ \nts\nts A^* }\bigr) (t)
   \frac{\dd \mu (t)}{\lvert \det (A) \rvert} \\[2mm] 
   & = \, \frac{\mu \bigl( 
   \widehat{g \circ \nts\nts A^*} \bigr)}
   {\lvert \det (A) \rvert}  \,
    =\, \frac{\widehat{\mu} \bigl(g \circ \nts\nts A^* \bigr)}
   {\lvert \det (A) \rvert}  \, = \,
   \frac{\bigl(A^* \! . \pts \widehat{\mu} \bigr)
    (g)}{\lvert \det (A) \rvert} \pts ,
\end{split}
\]
which implies the first claim. The second is now a consequence
of Eq.~\eqref{eq:forward-ac}.
\end{proof}

\subsection{Riesz products}

Of particular interest in the context of singular measures are
measures that have a representation as infinite Riesz products. Let us
recall one paradigmatic example of pure point type, and then
generalise it. Here, an expression of the form
$\prod_{m\geqslant 0} f_m (k)$ with continuous functions $f_m$ is a
short-hand for the measure that is defined as the vague limit of a
sequence of absolutely continuous measures, the latter being given by
the Radon--Nikodym densities $\prod_{m=0}^{n} f_m (k)$ with
$n\geqslant 0$.

\begin{lemma}\label{lem:2-Riesz}
  As a relation between translation-bounded measures on\/ $\RR$, one
  has
\[
     \prod_{m\geqslant 0} \bigl( 1 + \cos (2 \pi \pts 2^m k) \bigr)
     \, = \, \delta^{\pa}_{\ZZ} \pts ,  
\]
  where\/ $k\in\RR$ and convergence is in the vague topology.
\end{lemma}

\begin{proof}
  We employ a method that is well known from the theory of
  Bernoulli convolutions; compare \cite{PSS}.
  Define $\mu = \delta^{\pa}_{0} + \delta^{\pa}_{1}$ and consider
  $\nu = \frac{1}{2} \, \mu * \widetilde{\mu} = \delta^{\pa}_{0} +
  \frac{1}{2} (\delta^{\pa}_{1} + \delta^{\pa}_{-1})$, so
\[
     \widehat{\nu} (k) \, = \, 1 + \cos (2 \pi k) \pts .
\]   
With $f(x) = 2 x$, one has
$f \! . \pts \nu = \frac{1}{2} \, ( f \! .\mu) * ( f \! .
\widetilde{\mu} \pts )$ by Lemma~\ref{lem:push-conv},
where $ f\! . \widetilde{\mu} = \widetilde{f \! . \mu\,}\nts$.
Moreover, one has $\mu * \nts (f \! .\mu) * \ldots * \nts
(f^{n-1} \! .\mu) = \sum_{\ell=0}^{2^n - 1} \delta^{\pa}_{\ell}$.
Now, for $n\geqslant 1$, a simple convolution calculation gives
\[
    \Conv_{m=0}^{n-1} f^m \! . \pts \nu \, = \sum_{\ell=1-2^n}^{2^n -1} \!
    \myfrac{2^n - \lvert \ell \rvert}{2^n} \, \delta^{\pa}_{\ell}
    \; \xrightarrow{\, n \to \infty \,} \; \delta^{\pa}_{\ZZ} \pts ,
\]  
  with convergence in the vague topology.
  
  With $\widehat{f^m_{\vphantom{I}} \! . \pts \nu} = \widehat{\nu}
  \circ \nts f^m$, which follows from Lemma~\ref{lem:linear-FT},
  an application of the convolution theorem in conjunction with the
  continuity of the Fourier transform leads to
\begin{equation}\label{eq:Poisson}
    \widehat{\Conv_{m=0}^{n-1} f^m \! . \pts \nu }  \, = 
    \prod_{m=0}^{n-1} (\widehat{\nu} \circ \nts f^m )
     \; \xrightarrow{\, n \to \infty \,} \; 
    \widehat{\pts\delta^{\pa}_{\ZZ}\pts}
     \, = \, \delta^{\pa}_{\ZZ} \pts , 
\end{equation}
where the last step is the Poisson summation formula (PSF); compare
\cite[Prop.~9.4]{TAO}. Our claim now follows via the observation that
$\bigl( \widehat{\nu} \circ \nts f^m \bigr) (k) = 1 + \cos( 2 \pi \pts
2^m k) $.
\end{proof}  

More generally, let $2 \leqslant M \in \NN$ be fixed and consider
\[
    \mu \, = \sum_{\ell=0}^{M-1} \delta^{\pa}_{\ell}
    \quad \text{and} \quad  
    \nu \, = \, \frac{\mu * \widetilde{\mu}}{M}
    \, = \, \delta^{\pa}_{0} \, + \sum_{\ell=1}^{M-1}
    \myfrac{M - \ell}{M} \bigl( \delta^{\pa}_{\ell} + 
    \delta^{\pa}_{-\ell}\bigr) .
\]
With $f(x) = Mx$, one has
$\mu * \nts (f \! .\mu) * \ldots * \nts (f^{n-1} \! .\mu) =
\sum_{\ell=0}^{M^n - 1} \delta^{\pa}_{\ell}$ and
\[
     \Conv_{m=0}^{n-1} f^m \! . \nu \, = 
     \sum_{\ell=1-M^n}^{M^n - 1} \!
     \myfrac{M^n - \lvert \ell \rvert}{M^n} \, \delta^{\pa}_{\ell}
     \; \xrightarrow{\, n \to\infty \,} \; \delta^{\pa}_{\ZZ}
\]
in the vague topology, so that Eq.~\eqref{eq:Poisson} holds here
as well. Observe that
\[
     \widehat{\nu} (k) \, = \, 1  + 2 \sum_{\ell=1}^{M-1}
     \myfrac{M-\ell}{M} \, \cos ( 2 \pi \ell k) \pts ,
\]
which satisfies $\widehat{\nu} (k) \geqslant 0$ and
$\int_{0}^{1} \widehat{\nu} (k) \dd k = 1$. Moreover,
$\prod_{m=0}^{n-1} \bigl( \widehat{\nu} \circ \nts f^m \bigr)$ defines
a probability density on $[0,1]$ for each $n\in\NN$. Now, the
generalisation of Lemma~\ref{lem:2-Riesz} reads as follows.

\begin{prop}\label{prop:M-Riesz}
    For any\/ $2 \leqslant M \in \NN$, one has
\[
      \prod_{m\geqslant 0}  \Bigl( 1 + \pts 2 \nts \sum_{\ell=1}^{M-1}
      \myfrac{M-\ell}{M} \, \cos (2 \pi \ell M^m k) \Bigr) \, = \, 
      \delta^{\pa}_{\ZZ}  \pts ,
\]    
   where\/ $k\in\RR$ and convergence is in the vague topology. 
   \qed
\end{prop}

It is clear how to extend this to more than one dimension, the details
of which are left to the interested reader.

\section{Uniform distribution and averages}\label{sec:sample}

While uniform distribution results are usually stated for one
dimension, many of them have natural, though less well-known,
generalisations to higher dimensions. We shall need some of them to
calculate limits of various Birkhoff sums in our examples.  To
formulate the results, we represent $\TT^d= \RR^d /\ZZ^d$ as the
half-open unit cube $[0,1)^d$ with \mbox{(coordinate{\pts}-wise)}
addition modulo~$1$. As before, we use $\langle x \pts | \pts y
\rangle = \sum_{i=1}^{d} x_i \pts y_i$ for the standard inner product
in $\RR^d$. Let us recall some useful properties of non-singular
linear forms.

\begin{fact}\label{fact:null-lift}
  Consider a non-singular linear form\/
  $f \! : \, \RR^d \xrightarrow{\quad} \RR$, which can thus be written
  as\/ $f (x) = \langle a \pts | \pts x \rangle$ with
  $0 \ne a \in \RR^d$. Then, if\/ $\cE $ is a Lebesgue null set in\/
  $\RR$, its preimage\/ $f^{-1} (\cE)$ is a Lebesgue null set in\/
  $\RR^d$.
\end{fact}

\begin{proof}
  Let $\mu^{\pa}_{\mathrm{L}}$ and $\nu^{\pa}_{\mathrm{L}}$ denote
  Lebesgue measure in $\RR^d$ and $\RR$, respectively. Clearly, the
  linear mapping $f$ is differentiable, with
  $\nabla \nts f (x) = a \ne 0$ for all $x\in\RR^d$, hence certainly
  measurable and surjective. Now, the pushforward
  $f \! .  \mu^{\pa}_{\mathrm{L}}$ defines a regular Borel measure on
  $\RR$, with
  $\bigl( f \! . \mu^{\pa}_{\mathrm{L}}\bigr) (\cE) =
  \mu^{\pa}_{\mathrm{L}} \bigl( f^{-1} (\cE)\bigr)$
  for \emph{any} Borel set $\cE\subseteq\RR$; compare
  \cite[Thm.~39.C]{Halmos}. Due to the linearity of $f$, for any
  $t\in\RR$, we have $f^{-1} (t + \cE) = z^{\pa}_{t} + f^{-1} (\cE)$
  for some $z^{\pa}_{t} \in \RR^d$ with $f (z^{\pa}_{t}) = t$, which
  covers the empty set via the standard convention
  $x + \varnothing = \varnothing$.
   
  This property implies
  $\bigl( f\! . \mu^{\pa}_{\mathrm{L}}\bigr) ( t+\cE) = \bigl( f\!
  . \mu^{\pa}_{\mathrm{L}}\bigr) ( \cE)$
  for all $t\in\RR$ and all Borel sets $\cE$, which means that
  $f\! . \mu^{\pa}_{\mathrm{L}}$ is translation invariant and thus a
  multiple of Haar measure on $\RR$.  Consequently, we have
  $f\! . \mu^{\pa}_{\mathrm{L}} = c \pts\pts \nu^{\pa}_{\mathrm{L}}$,
  where $c>0$ follows from $a\ne 0$.  This means that
  $f\! . \mu^{\pa}_{\mathrm{L}}$ and $\nu^{\pa}_{\mathrm{L}}$ are
  equivalent as measures, and our claim on the Lebesgue null sets
  follows.
\end{proof}

\begin{fact}\label{fact:lin-dist}
  Let\/ $\alpha \in \RR$ with\/ $\lvert \alpha \rvert > 1$ be given
  and let $f$ be the linear form from
  Fact~\textnormal{\ref{fact:null-lift}}.  Then, for Lebesgue-a.e.\
  $x\in\RR^d$, the sequence\/ $\bigl(f (\alpha^n x)\bigr)_{n\in\NN}$
  is uniformly distributed modulo\/ $1$.
\end{fact}

\begin{proof}
  One has $f(\alpha^n x) = \alpha^n \pts t$ with
  $t=\langle a \pts | \pts x \rangle$ and $a\ne 0$. Clearly,
  $(\alpha^n \pts t)^{\pa}_{n\in\NN}$ is uniformly distributed modulo
  $1$ for a.e.\ $t\in\RR$ by standard results from uniform
  distribution theory; compare \cite[Thm.~4.3 and Exc.~4.3]{KN}. If
  $\cE$ is the corresponding null set of exceptional points, uniform
  distribution modulo $1$ of $\bigl(f (\alpha^n x)\bigr)_{n\in\NN}$
  fails precisely for all $x\in f^{-1} (\cE) \subset \RR^d$. By
  Fact~\ref{fact:null-lift}, $f^{-1} (\cE)$ is a null set in $\RR^d$,
  which implies the claim.
\end{proof}

\begin{lemma}\label{lem:d-equi}
  Let\/ $\alpha \in \RR$ with\/ $\lvert \alpha \rvert > 1$ be fixed.
  Then, for Lebesgue-a.e.\ $x\in\RR^d$, the sequence\/
  $(\alpha^n x)^{\pa}_{n\in\NN}$ taken modulo\/ $1$ is uniformly
  distributed in\/ $\TT^d$.
\end{lemma}

\begin{proof}
  For $d=1$, this is a well-known result from metric equidistribution
  theory \cite[Ch.~4]{KN}, as mentioned earlier.  For $d>1$ and any
  given $x\in\RR^d$, it is convenient to employ Weyl's criterion
  \cite[Thm.~6.2]{KN} and consider the convergence behaviour of
  character sums. In fact, this implies that uniform distribution of
  $(\alpha^n x)^{\pa}_{n\in\NN}$ modulo~$1$ is equivalent to uniform
  distribution modulo~$1$ of the sequences
  $(\alpha^n \langle k \pts | \pts x \rangle)^{\pa}_{n\in\NN}$ for all
  $k \in \ZZ^d \setminus \{ 0 \}$; compare \cite[Thm.~6.3]{KN}.  For
  each such $k$, let $\cE_k$ be the exceptional set of points
  $x\in\RR^d$ where uniform distribution fails, which is a null set by
  Fact~\ref{fact:lin-dist}.  Since $\ZZ^d \setminus \{ 0 \}$ is
  countable, the set $\bigcup_{k\in\ZZ^d \setminus \{ 0 \}} \cE_k$ is
  still a null set in $\RR^d$, and the claim follows.
\end{proof}

Next, we need to understand averages of various types of periodic and
almost periodic functions, in particular along exponential sequences
of the above type.

\begin{lemma}\label{lem:trig-poly}
   Let\/ $\alpha \in \RR$ with\/ $\lvert \alpha \rvert > 1$ be fixed.
   For any\/ $a\in\RR^d$ and then a.e.\ $x\in\RR^d$, one has
\[
   \lim_{N\to\infty} \myfrac{1}{N} \sum_{n=0}^{N-1}
   \ee^{2 \pi \ii \pts \alpha^n \langle a \pts | \pts x \rangle}
   \, = \, \delta^{\pa}_{a,0} \pts .
\]
\end{lemma}

\begin{proof}
  When $a=0$, the limit is $1$ for \emph{all} $x\in\RR^d$, so let
  $a\ne 0$. Then, by Fact~\ref{fact:lin-dist},
  $(\alpha^n \langle a \pts | \pts x \rangle)^{\pa}_{n\in\NN}$ is
  uniformly distributed modulo $1$ for a.e.\ $x\in\RR^d$, where the
  null set $\cE_a$ of exceptions depends on $a$.  So, for any given
  $a$ and then every $x \in\RR^d \setminus \cE_a$, we get
\[
   \lim_{N\to\infty} \myfrac{1}{N} \sum_{n=0}^{N-1}
   \ee^{2 \pi \ii \pts \alpha^n \langle a \pts | \pts x \rangle}
   \, = \int_{0}^{1} \ee^{2 \pi \ii t} \dd t \, = \, 0
\]
  by Weyl's lemma.
\end{proof}

The next step is an extension to (complex) trigonometric polynomials,
as given by
\[
    P_m (x) \, = \, c^{\pa}_{\pts 0} \, + \sum_{\ell=1}^{m}
    c^{\pa}_{\ell} \, \ee^{2 \pi \ii \pts \langle k_{\ell} \pts | \pts x \rangle}
\]
with $m\nts\in\nts\NN_0$ and coefficients $c^{\pa}_{\ell}\nts \in\nts \CC$. When
$m\nts\geqslant\nts 1$, the frequency vectors
$k^{\pa}_{1}, \ldots , k^{\pa}_{m}$ are assumed to be non-zero and
distinct. Clearly, under the conditions of Lemma~\ref{lem:trig-poly},
one obtains
\begin{equation}\label{eq:trig-poly}
   \lim_{N\to\infty} \myfrac{1}{N} \sum_{n=0}^{N-1}
   P_m (\alpha^n x)   \, = \, c^{\pa}_{\pts 0} \, = \, \MM (P_m)
\end{equation}
for a.e.\ $x\in\RR^d$. Here, $\MM (f)$ is the \emph{mean}
of a bounded function,
\begin{equation}\label{eq:def-mean}
   \MM (f) \, := \, \lim_{n\to\infty} \myfrac{1}{ \vol (A_n)}
   \int_{A_n} f(x) \dd x \pts ,
\end{equation}
where $\cA = ( A_n )^{\pa}_{n\in\NN}$ is a fixed sequence of growing
sets for the averaging process. The sets $A_n$ are supposed to be
sufficiently `nice', which means that one assumes a property of
F{\o}lner or van Hove type.  To be concrete, we can think of $A_n$ as
the closed cube of sidelength $n$ centred at $0$. It is clear that the
limit in \eqref{eq:def-mean} exists for trigonometric
polynomials. More generally, it exists for all functions that are
\emph{uniformly almost periodic}, which are often also called Bohr
almost periodic. They are the continuous functions that can uniformly
be approximated by trigonometric polynomials. In other words, the
space of uniformly almost periodic functions is the
$\|.\|^{\pa}_{\infty}$-closure of the space of trigonometric
polynomials; see \cite{Cord} for general results.

\begin{prop}\label{prop:BohrMean}
  Let\/ $f\! : \, \RR^d \xrightarrow{\quad} \CC$ be a uniformly $($or
  Bohr$\, )$ almost periodic function, and let\/ $\alpha \in \RR$
  with\/ $\lvert \alpha \rvert > 1$ be given. Then, for a.e.\
  $x\in\RR^d$, one has
\[
     \lim_{N\to\infty} \myfrac{1}{N} \sum_{n=0}^{N-1}
     f (\alpha^n x) \, = \, \MM (f) \pts .
\]   
  In particular, this applies to functions of the form\/
  $f = \log (g)$ with\/ $g$ a non-negative, uniformly
  almost periodic function that is bounded away from\/ $0$.
\end{prop}

\begin{proof}
  The first claim for $d=1$ is \cite[Thm.~6.4.4]{BHL}. A close
  inspection of its proof reveals that the same chain of arguments
  also applies to the case $d>1$, which is all we need here.
  
  The second claim follows from the first because
  $g(x) \geqslant \delta > 0$ for all $x\in \RR^d$ implies that
  $\log (g)$ is again a uniformly almost periodic function
  \cite[Fact~6.14]{BFGR}.
\end{proof}

In the attempt to generalise Proposition~\ref{prop:BohrMean} beyond
uniformly almost periodic functions, one difficulty emerges when $f$
is no longer locally Riemann-integrable.  Let us first look at
periodic functions, where we begin by recalling a classic result.

\begin{fact}[{\cite[Lemma~6.3.3]{BHL}}]\label{fact:birk}
  Let\/ $q\in\ZZ$ with\/ $\lvert q \rvert \geqslant 2$ be fixed, and
  consider a function\/ $f\in L^{1}_{\mathrm{loc}} (\RR)$ that is\/
  $1$-periodic. Then,
\[
   \myfrac{1}{N} \sum_{n=0}^{N-1} f (q^n x) \,
   \xrightarrow{\, N \to \infty \,} \int_{0}^{1} \!
   f(y) \dd y \, = \, \MM (f)
\]
   holds for a.e.\ $x\in\RR$.  \qed
\end{fact}

The key ingredient to Fact~\ref{fact:birk} is the ergodicity of
Lebesgue measure on $\TT$ for the dynamical system defined by
$x \mapsto q x$ modulo $1$, which permits to use Birkhoff's ergodic
theorem instead of Weyl's lemma and uniform distribution of
$(q^n x)^{\pa}_{n\in\NN}$ for a.e.\ $x\in\RR$. The natural counterpart
on $\TT^d$ can be stated as follows.

\begin{lemma}\label{lem:Q-sample}
  Let\/ $Q$ be a non-singular endomorphism of\/ $\TT^d$ such that
  no eigenvalue is a root of unity, and consider a\/ $\ZZ^d$-periodic
  function\/ $f \in L^{1}_{\mathrm{loc}} (\RR^d)$. Then, for a.e.\
  $x\in\RR^d$, one has
\[
   \myfrac{1}{N} \sum_{n=0}^{N-1} f (Q^n x) \,
   \xrightarrow{\, N \to \infty \,} \int_{\TT^d}
   f(y) \dd y \, = \, \MM (f) \pts .
\]
  In particular, this result applies to every toral endomorphism 
  that is expansive.
\end{lemma}

\begin{proof}
  Under our assumptions, Lebesgue measure is an invariant and ergodic
  measure for the dynamical system defined by $Q$ on $\TT^d $; see
  \cite[Cor.~2.20]{EW}. The main statement now follows from Birkhoff's
  ergodic theorem. Since all eigenvalues of an expansive
  $Q\in \mathrm{End} (\TT^d)$ satisfy $\lvert \lambda \rvert > 1$, the
  last claim is clear.
\end{proof}

Beyond Fact~\ref{fact:birk} and Lemma~\ref{lem:Q-sample}, we will need
the following result, which can be viewed as a variant of Sobol's
theorem \cite{Sobol}; see also \cite{Hart,BHL}.

\begin{lemma}\label{lem:gen-Sobol}
  Let\/ $p \geqslant 0$ be a trigonometric polynomial in\/ $d$
  variables, and let\/ $\alpha \in \RR$ with\/
  $\lvert \alpha \rvert >1$ be fixed. Let us further assume that, for
  some sufficiently small\/ $\delta >0$, the critical points of\/ $p$
  with value in\/ $[0,\delta]$ are isolated.  Then, for
  Lebesgue-a.e.\/ $x\in\RR^d$, one has
\[
    \lim_{N\to\infty} \myfrac{1}{N} \sum_{n=0}^{N-1}
    \log \bigl( p ( \alpha^n x) \bigr) \, = \,
    \MM \bigl( \log (p) \bigr) .
\]
\end{lemma}

\begin{proof}[Sketch of proof]
  Since the case $p(k)\geqslant \delta>0$ for all $k\in\RR^d$ is
  covered by Proposition~\ref{prop:BohrMean}, we assume
  $\inf_{k\in\RR^d} p(k) = 0$ and thus
  $\inf_{k\in\RR^d} \log (p(k)) = -\infty$, which is the origin of the
  complication. Note, however, that all singularities of $\log ({p})$
  are of logarithmic type and hence locally integrable, so $\log (p)$
  is no longer uniformly, but still Stepanov almost periodic; compare
  \cite[pp.~356--359]{BHL} as well as \cite[Sec.~VI.4]{Cord}.

  Now, we have to deal with the small local minima of $p$.  By
  assumption, there is a $\delta>0$ such that the points $k$ with
  $\nabla\nts p (k) =0$ and $p(k) \in [0,\delta]$ are isolated. As $p$
  is a quasiperiodic function, the set of critical points of this
  type, $Z$ say, must then be uniformly discrete.

  Now, with a Borel--Cantelli argument, compare \cite[Thm.~6.3.5]{BHL}
  and \cite[Prop.~5.1]{BaGriMa}, one can derive that, for a.e.\
  $x\in\RR^d$, the sequence $(\alpha^n x)^{\pa}_{n\in\NN_0}$ stays
  sufficiently far away from $Z$ so that the average via the Birkhoff
  sum ultimately is not distorted by the singularities or almost
  singularities of $\log (p)$, and Sobol's theorem can be
  applied. This gives
\[
     \lim_{N\to\infty} \myfrac{1}{N} \sum_{n=0}^{N-1}
    \log \bigl( {p} ( \alpha^n x) \bigr) \, = \,
    \MM \bigl( \log (p) \bigr)
\]
for a.e.\ $x\in\RR^d$ as claimed.
\end{proof}

At this point, we are set to start the spectral analysis of inflation
systems via their pair correlations, where we begin with the theory in
one dimension.

\section{Results in one dimension}\label{sec:1d}

Let us recall the situation in one dimension from
\cite{BG15,BGM}. Consider a \emph{primitive substitution} $\varrho$ on
an $L$-letter alphabet
$\mathcal{A}=\{a^{\pa}_{1},\dots,a^{\pa}_{\nts L}\}$.  It defines a
unique \emph{symbolic hull} $\XX_{\varrho}$, which is compact and
consists of a single local indistinguishability (LI) class. This hull
can be constructed as the closure of the shift orbit of a two-sided
fixed point of a suitable power of $\varrho$. This shift space gives
rise to a uniquely (in fact, strictly) ergodic dynamical system under
the $\ZZ$-action of the shift, denoted as $(\XX_{\varrho},\ZZ)$.

The corresponding \emph{substitution matrix} $M$ is the primitive
non-negative $L \! \times \! L$-matrix with elements
$M_{ij}=\card_{a^{\pa}_{i}} \bigl(\varrho(a^{\pa}_{j})\bigr)$ and
Perron--Frobenius (PF) eigenvalue $\lambda>1$.  The matching (properly
normalised) right eigenvector of $M$ encodes the letter frequencies,
while the left eigenvector determines the ratios of natural tile
lengths for a consistent geometric \emph{inflation rule}. The latter
acts on $L$ intervals (which are our prototiles), one for each letter,
of lengths corresponding to the entries of the left eigenvector. If
the $L$ intervals do not have distinct lengths, we distinguish
congruent ones by labels (or colours). The inflation map induced by
$\varrho$ then consists of a scaling of the intervals by the inflation
multiplier $\lambda$ and their subsequent dissection into original
prototiles, according to the order determined by the substitution rule
$\varrho$. In this setting, the inflation again defines a strictly
ergodic dynamical system, now (in general) under the continuous
translation action of $\RR$, denoted as $(\YY,\RR)$, with $\YY$ the
new \emph{tiling hull}.

To capture the geometric information, let us collect the relative
positions of the tiles in the inflation map in a set-valued
\emph{displacement matrix} $T$. Each element $T_{ij}$ thus is a set,
viewed as a list of length $M_{ij}$ that contains the relative
positions of the interval (or tile) of type $i$ in the inflated
interval (or supertile) of type $j$ (and is the empty set if
$M_{ij}=0$). To define the distance between tiles, we assign a
reference point to each tile, which we usually choose to be the left
endpoint of the interval. Clearly, since the reference point
determines the tile and its position, the set of (labelled or
coloured) reference points is \emph{mutually locally derivable} (MLD)
with the tiling by intervals.  For a given tiling, define $\vL_{i}$ as
the set of all reference points of tiles of type $i$, and
$\vL=\dot{\bigcup}_{i=1}^{L}\vL_{i}$ as the set of all such reference
points. \smallskip

Let $\nu^{\pa}_{ij} (z)$ with $z \geqslant 0$ be the relative
frequency of the occurrence of a tile of type $i$ (left) and one of
type $j$ (right) at distance $z$, with the understanding that
$\nu^{\pa}_{ij} (-z) = \nu^{\pa}_{ji} (z)$. These are the \emph{pair
  correlation coefficients} of the inflation rule, which exist for all
elements of the hull and are independent of the choice of the element.
Given $\vL$, decomposed as $\vL = \dot{\bigcup}_{i} \vL_i$, one can
represent each coefficient as a limit,
\[
   \nu^{\pa}_{ij} (z) \, = \lim_{r\to\infty}
   \frac{\card \bigl(  B_{r} (0) \cap \vL_{i} \cap
   (\vL_{j} - z) \bigr)}
   {\card ( B_{r} (0) \cap \vL )} \, = \, 
   \frac{\dens\bigl(\vL_{i}\cap (\vL_{j}-z)\bigr)}{\dens(\vL)}
   \, \geqslant \, 0\pts .
\]
Due to the strict ergodicity, one has $\nu^{\pa}_{ij} (z) > 0$ if and
only if $z \in S_{ij} := \vL_{j} - \vL_{i}$, where the sets 
$S_{ij}$ are independent of the choice of
$\vL$ from the hull, because the latter is minimal and thus
consists of a single LI class \cite{TAO}. 

Let us now recall the general renormalisation relations for the 
$\nu^{\pa}_{ij}$ from \cite{BG15,BFGR,BaGriMa}, which are proved
in full generality in \cite{BGM}, also for higher dimensions; see
Eq.~\eqref{eq:renogen} below.

\begin{lemma}\label{lem:ren-eq}
  Let\/ $\nu^{\pa}_{ij}$ be the pair correlation coefficients of the
  geometric inflation rule induced by the primitive\/ $L$-letter
  substitution\/ $\varrho$ with inflation multiplier\/ $\lambda$, and
  let\/ $T$ be the corresponding set-valued displacement matrix. Then,
  they satisfy the identities
\[
    \nu^{\pa}_{ij} (z) \, = \, \myfrac{1}{\lambda} \sum_{m,n=1}^{L}
    \,\sum_{r\in T_{im}}  \,\sum_{s\in T_{jn}}
    \nu^{\pa}_{mn} \left( \myfrac{z+r-s}{\lambda} \right)
\]
  for arbitrary\/ $z\in \RR$.   \qed
\end{lemma}

\begin{remark}
  The identities of Lemma~\ref{lem:ren-eq} have a special structure,
  which we call an \emph{exact renormalisation} for the following
  reason. First, there is a finite subset of identities that close,
  and give what is known as the \emph{self-consistency} part of the
  identities. Then, all remaining relations are purely
  \emph{recursive}, which also implies that the solution space of the
  renormalisation identities is finite-dimensional. This is further
  discussed and explored in \cite{BG15,BGM}.  \exend
\end{remark}

Now, define
$\vU^{\pa}_{ij} = \sum_{z\in S_{ij}} \nu^{\pa}_{ij} (z) \,
\delta^{\pa}_{z}$,
which is a pure point measure for each $1 \leqslant i,j \leqslant L$.
For the measure vector
$\vU = ( \vU^{\pa}_{11}, \vU^{\pa}_{12}, \ldots , \vU^{\pa}_{\nts L L}
)$, we use $f \nts .\vU$ for the componentwise pushforward, where
$f(x) = \lambda x$ as before. With this, Lemma~\ref{lem:ren-eq}
implies the matching relation for the \emph{pair correlation measures}
to be
\[
     \vU \, = \, \myfrac{1}{\lambda}  \bigl( \pts
     \widetilde{\nts \delta^{\pa}_{T} \nts }
     \overset{*}{\otimes} \delta^{\pa}_{T} \bigr) * 
     ( f \nts .\vU ) \pts ,
\]
where $\delta^{\pa}_{T}$ is the measure-valued matrix with elements
$\delta^{\pa}_{T_{ij}}$ and $\overset{*}{\otimes}$ denotes the
Kronecker product of two measure-valued matrices
with convolution as multiplication.

All elements of $\vU$ are Fourier-transformable as measures, which
follows from \cite[Lemma~1]{BG15}.  Thus, we define the 
\emph{Fourier matrix} of our inflation system as
\[
      B (k) \, := \, \overline{\widehat{\delta^{\pa}_{T}}} (k) 
      \, = \, \widehat{\delta^{\pa}_{T}} (-k) \pts ,
\]
which is an $L \!\times \! L$ matrix function with trigonometric
polynomials as entries, and thus analytic in $k$. Now, by Fourier 
transform in conjunction with the convolution theorem, one finds
\begin{equation}\label{eq:ren-FT}
     \widehat{\vU} \, = \, \myfrac{1}{\lambda^2}
     \left( B (.) \otimes \overline{B (.)} \, \right)
     \bigl( f^{-1} \! . \pts \widehat{\vU} \pts \bigr) ,
\end{equation}
to be read as a relation between measure vectors.
The main advantage of this formulation is that we now actually
obtain \emph{three} equations from \eqref{eq:ren-FT} as follows.

Each $\widehat{\vU}_{ij}$ is a measure that has a unique decomposition
into a pure point (\textsf{pp}) and a continuous part, with a
countable supporting set for the pure point part. Taking the union of
the latter over all $i,j$ allows us to define the decomposition
\[
     \widehat{\vU} \, = \, \widehat{\vU}_{\mathsf{pp}}
     +  \widehat{\vU}_{\mathsf{cont}}
\]
with a matching decomposition $\RR = \cE_{\mathsf{pp}} \, \dot{\cup}
\, \cE_{\mathsf{cont}}$. Here, $\cE_{\mathsf{pp}}$ is a countable set,
and we may assume without loss of generality that it is also invariant
under $f$ and $f^{-1}$, for instance by replacing $\cE_{\mathsf{pp}}$
with $\bigcup_{n\in\ZZ} f^n (\cE_{\mathsf{pp}})$, which is still
countable. The complement then still is a valid supporting set for
the continuous part, and also invariant under $f$ and $f^{-1}$.

Repeating this type of argument, we can further split
$\widehat{\vU}_{\mathsf{cont}}$ into its singular continuous
(\textsf{sc}) and absolutely continuous (\textsf{ac}) component, which
goes along with a decomposition
$\RR = \cE_{\mathsf{pp}} \, \dot{\cup} \, \cE_{\mathsf{sc}} \,
\dot{\cup} \, \cE_{\mathsf{ac}}$,
where each supporting set is invariant under $f$ and $f^{-1}$; see
\cite{BGM} for a more detailed discussion of this point. This
decomposition leads to the following result.

\begin{lemma}\label{lem:separate}
   The measure vector\/ $\widehat{\vU}$ satisfies the three separate
   equations
\[
     \widehat{\vU}_{\alpha} \, = \,
      \myfrac{1}{\lambda^2}
     \left( B (.) \otimes \overline{B (.)} \, \right)
     \bigl( f^{-1} \! . \pts  \widehat{\vU}_{\alpha} \bigr) ,
\]
    for\/ $\alpha \in \{ \mathsf{pp}, \mathsf{sc}, \mathsf{ac} \}$.   
\end{lemma}

\begin{proof}
  This is a consequence of the fact that
  $B(k) \otimes \overline{ B(k)}$ is analytic in $k$, hence cannot
  change the spectral type, together with
  $\bigl( f^{-1} \! . \pts \widehat{\vU} \pts \bigr)_{\alpha} = f^{-1}
  \! .  \widehat{\vU}_{\alpha}$
  due to $f$ being a simple dilation, which cannot change the spectral
  type either. The claim now follows from restricting
  Eq.~\eqref{eq:ren-FT} to the supporting sets $\cE_{\alpha}$
  constructed above.
\end{proof}

All three equations have interesting implications, as discussed in
\cite{BG15,BFGR,BGM}. Here, we concentrate on the \textsf{ac} part. To
get some insight into the latter, we denote the Radon--Nikodym density
vector of $\widehat{\vU}_{\mathsf{ac}}$ by $h$.  Then,
Lemma~\ref{lem:separate} results in the relation
\[
    h(k) \, = \, \myfrac{1}{\lambda} \,
    \bigl( B(k) \otimes \overline{B(k)}\pts \bigr)
    h (\lambda k) \pts ,
\]
which has to hold for a.e.\ $k\in\RR$ and can be iterated. Note that
the different power of $\lambda$ in the denominator in comparison to
Lemma~\ref{lem:separate} results from a change of variable
transformation. For values of $k$ with $\det ( B(k)) \ne 0$, it can
also be inverted to get an iteration in the opposite direction.  It is
a crucial observation from \cite{BFGR,BaGriMa,BGM} that the asymptotic
behaviour can be analysed from the simpler iterations
\begin{equation}\label{eq:v-iter}
   v(k) \, = \, \myfrac{1}{\mbox{\small $\sqrt{\lambda}$}\pts}
   \, B(k) \pts v (\lambda k) \quad \text{and} \quad
   v(\lambda k) \, = \mbox{\small $\sqrt{\lambda}$}
   \, B^{-1} (k) \pts v(k) \pts ,
\end{equation}
where the components of $v(k)$ are locally square integrable
functions. Using Fact~\ref{fact:decompose}, this emerges from a
decomposition of $\bigl(h_{ij} (k) \bigr)$, viewed as a positive
semi-definite Hermitian matrix, as a sum of rank-$1$ matrices of the
form $v^{}_i(k) \, {v^{\dagger}_j (k)}$ and the observation that the
overall growth rate is dictated by the maximal growth rate of these
summands; see \cite{BFGR,BGM} for details.

To capture the asymptotic behaviour, one defines the \emph{Lyapunov
  exponents}, compare \cite{Viana}, for the iterations that emerge
from Eq.~\eqref{eq:v-iter}, which is possible when $B(k)$ is
invertible for a.e.\ $k\in\RR$.  It turns out that the required values
can all be related to the extremal Lyapunov exponents of the matrix
cocycle defined by
\begin{equation}\label{eq:def-cycle}
    B^{(n)} (k) \, := \, B(k) \pts B(\lambda k) \cdots
    B(\lambda^{n-1} k) \pts ,
\end{equation}
which happens to be the Fourier matrix of $\varrho^n$.  The quantities
of interest to us here are controlled by the \emph{maximal} Lyapunov
exponent of this cocycle, defined as
\begin{equation}\label{eq:def-L}
   \chi^{B} (k) \, := \, \limsup_{n\to\infty}
   \myfrac{1}{n} \log \big\| B^{(n)} (k) \big\| ,
\end{equation}
where $\| . \|$ refers to any sub-multiplicative matrix norm, such as
the spectral norm or the Frobenius norm. In favourable cases,
$\chi^{B} (k)$ will exist as a limit for a.e.\ $k\in\RR$, as we shall
see later in several examples.  The main criterion can now be
formulated as follows.

\begin{theorem}\label{thm:1D}
  Let\/ $\varrho$ be a primitive substitution on a finite alphabet,
  and consider the corresponding inflation rule with inflation
  multiplier\/ $\lambda = \lambda^{\pa}_{\mathrm{PF}}$ for intervals of
  natural length. Let\/ $\chi^{B} (k)$ be the maximal Lyapunov
  exponent of the Fourier matrix cocycle \eqref{eq:def-cycle}, and
  assume that\/ $\det \bigl( B(k) \bigr) \ne 0$ for at least one\/
  $k\in\RR$.
  
  If there is some\/ $\varepsilon > 0$ such that\/
  $\chi^{B} (k) \leqslant \frac{1}{2}\log (\lambda) - \varepsilon$
  holds for Lebesgue-a.e.\ $k\in\RR$, one has\/
  $\widehat{\vU}_{\mathsf{ac}} = 0$, and the diffraction measure of
  the system is singular.
\end{theorem}

\begin{proof}[Sketch of proof]
  Under our assumptions,\footnote{Since $\det \bigl( B(k)\bigr)$ is a
    trigonometric polynomial, it is either identically $0$ or has
    isolated zeros.}  for a.e.\ $k\in\RR$ with $k\ne 0$, the sequence
  $\bigl( h (\lambda^n k) \bigr)_{n\in\NN}$ of Radon--Nikodym density
  vectors, as $n\to\infty$, displays an exponential growth of order
  $\ee^{2 (D-\delta) n}$, where
  $D = \frac{1}{2} \log (\lambda) - \chi^{B} (k) \geqslant \varepsilon
  > 0$
  and $\delta>0$ can be chosen such that $D-\delta>0$. The implied
  constant will depend on $k$ and $\delta$.  Such a behaviour is
  incompatible with the translation-boundedness of the components of
  $\widehat{\vU}_{\mathsf{ac}}$, which is a contradiction unless
  $h (k) = 0$ for a.e.\ $k\in\RR$, hence
  $\widehat{\vU}_{\mathsf{ac}} = 0$.  For further details, we refer to
  \cite[Sec.~6.7 and App.~B]{BFGR} as well as to the general treatment
  in \cite{BGM}.
\end{proof}

\begin{remark}\label{rem:more}
  The statement of Theorem~\ref{thm:1D} can be strengthened and
  extended in various ways. First of all, one can show that
  $\chi^{B}(k)\leqslant \log\sqrt{\lambda}$ holds for a.e.\ $k\in\RR$.
  As a consequence, a non-trivial \textsf{ac} diffraction component is
  only possible when $\chi^{B}(k) = \log\sqrt{\lambda}$ is true for
  $k$ in a subset of positive measure in every interval of the form
  $[-\lambda a, -a]$ or $[a, \lambda a]$ with $a>0$. When $\lambda$ is
  a PV number without any further restriction, which thus also covers
  all primitive inflation rules of constant length as well as those
  with integer inflation factor, the relation must even hold for
  Lebesgue-a.e.\ $k\in\RR$; see \cite{BGM} for details.  This
  poses severe restrictions on the existence of \textsf{ac}
  diffraction in inflation systems beyond the necessary criterion of
  Berlinkov and Solomyak \cite{BS}.  \exend
\end{remark}

\section{Consequences and an application}\label{sec:ex-no-ac}

For the Fibonacci inflation, the exact renormalisation for the pair
correlation functions was used to establish a spectral purity result
and then pure point spectrum \cite{BG15}, thus confirming a known
property in an independent way. The same line of thought works for all
noble means inflations in complete analogy.

It is tempting to expect a similar result for all irreducible PV
inflations, but one quickly realises that spectral purity is
essentially equivalent to almost everywhere injectivity of the factor
map onto the maximal equicontinuous factor (MEF). While the existence
of non-trivial point spectrum in one-dimensional inflation tilings
requires $\lambda$ to be a PV number \cite{Boris}, it is the exclusion
of any continuous spectral component that would settle the (still
open) Pisot substitution conjecture.

A less ambitious task thus is to establish the mere absence of
absolutely continuous diffraction or spectral measures.  It has long
been `known' (without mathematical proof) that the presence of
\textsf{ac} diffraction requires a particular scaling property of the
diffraction measure as a function of the system size.  This stems from
the heuristic expectation that a structure has an \textsf{ac}
diffraction spectrum if its fluctuations are somewhat similar to those
of a disordered random structure, so fluctuations growing as
$\sqrt{N}$ for a chain of length $N$, in line with the law of large
numbers. This behaviour corresponds to a wandering exponent equal to
$\frac{1}{2}$; see \cite{Aubry,GL,Luck} for an application to
aperiodic structures.

In the case of constant-length substitutions, this effectively
corresponds to a condition on the spectrum of the substitution matrix
$M$. Namely, if $\lambda$ is its PF eigenvalue, $M$ must also have an
eigenvalue $\sqrt{\lambda}$ or one of that modulus.  The necessity of
an eigenvalue of modulus $\sqrt{\lambda}$ for the existence of an
\textsf{ac} spectral measure was recently proved in \cite{BS}. That
this criterion is necessary, but not sufficient, can be shown by an
example, for instance using the constant-length substitution
\begin{equation}\label{eq:def-gegen}
    a \mapsto ab \, , \quad b \mapsto ca \, , \quad c \mapsto bd \, ,
    \quad d \mapsto dc
\end{equation}
on the $4$-letter alphabet $\{ a,b,c,d \}$. The substitution matrix
reads
\[
     M \, = \, \begin{pmatrix}
     1 & 1 & 0 & 0 \\ 1 & 0 & 1 & 0 \\
     0 & 1 & 0 & 1 \\ 0 & 0 & 1 & 1 \end{pmatrix}
\]
and has spectrum $\{ 2, \pm \sqrt{2}, 0 \}$, hence clearly satisfies
the $\sqrt{\lambda}\,$-criterion. Nevertheless, as was shown in
\cite{CG} on the basis of Bartlett's algorithmic classification of
spectral types \cite{Bart}, all spectral measures of this substitution
are singular.

Let us apply Lyapunov exponents to reach this conclusion in an
independent way. It is straight-forward to calculate
\[
   T \, = \, \begin{pmatrix}
   \{ 0 \} & \{ 1 \} & \varnothing & \varnothing \\
   \{ 1 \} & \varnothing & \{ 0 \} & \varnothing \\
   \varnothing & \{ 0 \}  & \varnothing & \{ 1 \}\\
   \varnothing & \varnothing & \{ 1 \} & \{ 0 \} \end{pmatrix}
   \quad \text{and} \quad
   B (k) \, = \, \begin{pmatrix}
   1 & z & 0 & 0 \\ z & 0 & 1 & 0 \\
   0 & 1 & 0 & z \\ 0 & 0 & z & 1 \end{pmatrix}
\]
where $z = \ee^{2 \pi \ii k}$.  One has $\det \bigl( B(k) \bigr) = z^4
- 1$ which vanishes only for $k\in\frac{1}{4} \ZZ$, so that $B(k)$ is
invertible for a.e.\ $k\in \RR$. Let us now, for $n\in\NN$, define the
matrices
\begin{equation}\label{eq:cocycle-def}
    B^{(n)} (k) \, := \,
    B(k) \pts B(2 k) \pts B(4 k)\cdots B(2^{n-1} k) \pts .
\end{equation}
By definition, $B^{(1)} = B$ is the Fourier matrix of $\varrho$, while
$B^{(n)}$ is the Fourier matrix of $\varrho^n \nts$, and hence a
natural object to study in this context.\footnote{Notice that, while the
  Fourier matrices $B(k)$ for different $k$ do generally not commute,
  the matrices $B^{(2)}(k)=B(k)\pts B(2k)$, which correspond to the
  square of the substitution rule \eqref{eq:def-gegen}, in fact form a
  commuting family of matrices. This corresponds to the fact that the
  substitution is non-Abelian in the sense of \cite{Q}, meaning that
  the column-wise letter permutations do not commute, while its
  square becomes Abelian.}

Since the substitution is of constant length, $B^{(n)}$ defines a
cocycle over the compact dynamical system defined by $k \mapsto 2 k$
modulo $1$ on $\TT$. We thus have Oseledec's multiplicative ergodic
theorem \cite{Viana} at our disposal, which implies that the Lyapunov
exponents exist for a.e.\ $k\in\RR$ and satisfy forward Lyapunov
regularity, hence in particular sum to
\[
    \lim_{n\to\infty} \myfrac{1}{n} \sum_{\ell=0}^{n-1}
    \log \pts \bigl| \det \bigl( B(2^{\ell} k) \bigr) \bigr| \, = \,
    \MM \bigl( \log \pts \bigl| \det \bigl(B(.) \bigr) \bigr| \bigr)
    \, = \, \fm (z^4 - 1) \, = \, 0 \pts ,
\]
where the first equality is a consequence of Birkhoff's ergodic
theorem, as detailed in Fact~\ref{fact:birk}, while the last step
follows directly from Fact~\ref{fact:cyclo}.

To continue, it is helpful to observe that $B(k)$ admits a
$k$-independent splitting of $\CC^4$ into a two-dimensional and two
one-dimensional subspaces. Concretely, one finds
\[
  U B(k) \pts U^{-1} \, = \, \begin{pmatrix} 1 \! + \!
    z & 0 & 0 & 0 \\
    0 & 1\! - \! z & 0 & 0 \\ 0 & 0 & -z & 1 \\
    0 & 0 & 1 & z \end{pmatrix}
    \quad \text{with} \quad  U \, = \, \myfrac{1}{2} \begin{pmatrix}
    1 & 1 & 1 & 1 \\ 1 & -1 & -1 & 1 \\ 1 & -1 & 1 & -1 \\
    1 & 1 & -1 & -1 \end{pmatrix} ,
\]
where the unitary matrix $U$ is an involution, so $U^{-1} = U$.  By
standard arguments, it is now clear that two of the four exponents are
given by $\fm (1+z) = 0$ and $\fm (1-z) =0$, which derives from the
invariant one-dimensional subspaces. The remaining two exponents must
still sum to $0$, and can be determined from the induced cocycle
$\tilde{B}^{(n)} (k) = \tilde{B} (k) \pts \tilde{B} (2 k) \cdots
\tilde{B} (2^{n-1} k)$
with
$\tilde{B} (k) = \left( \begin{smallmatrix} -z & 1 \\ 1 &
    z \end{smallmatrix} \right)$.
With $p^{\pa}_{N} (k) := \| \tilde{B}^{(N)} (k) \|^{2}_{\mathrm{F}}$,
which is a trigonometric polynomial due to the use of the Frobenius
norm, we know that
\begin{equation}\label{eq:means}
   \chi^{B} (k) \, = \, \chi^{\tilde{B}} (k) \, \leqslant \,
   \myfrac{1}{N} \pts\pts \MM \bigl( \log \| \tilde{B}^{(N)} (k) 
   \|^{\pa}_{\mathrm{F}} \bigr)  \, = \, \myfrac{\fm (p^{\pa}_{N} )}{2 N}
   \, =: \, m^{\pa}_{N}
\end{equation}
holds for a.e.\ $k\in\RR$ and every $N\in\NN$. Then, one also 
has $\chi^{B} (k) \leqslant \liminf_{N\to\infty} m^{\pa}_{N}$.

\begin{table}
  \caption{Some values \label{tab:num} of the means 
    $m^{\protect \pa}_{N}$ from
    Eq.~\eqref{eq:means}, calculated via Eq.~\eqref{eq:Jensen}.
    The numerical error is always less than $10^{-3}$.\vspace{-1ex}}
\renewcommand{\arraystretch}{1.2}
\begin{tabular}{|c|c@{\;\;}c@{\;\;}c@{\;\;}c@{\;\;}c
      @{\;\;}c@{\;\;}c@{\;\;}c@{\;\;}c@{\;\;}c
      @{\;\;}c@{\;\;}c|}
\hline
$N$ & 1 & 2 & 3 & 4 & 5 & 6 & 7 & 8 & 9 & 10 & 11 & 12\\
\hline
$m^{\pa}_{N}$ & 
0.693 & 0.478 & 0.379 & 0.334 & 0.302 & 0.274 & 0.252 &
0.235 & 0.220 & 0.208 & 0.198 & 0.189\\[0.5mm]
\hline 
\end{tabular} 
\end{table}

Now, employing Jensen's formula again, the numbers $m^{\pa}_{N}$ can
easily be calculated numerically with high precision, and are given in
Table~\ref{tab:num} for $N\leqslant 12$. These values clearly show
that
$\chi^{B} (k) \leqslant \frac{1}{5} < \log\sqrt{2} \approx 0.346
{\pts} 574$, which implies the absence of \textsf{ac} diffraction.

Since we are in the constant-length case, this result translates into
one on the spectral measures via the general results of
\cite[Prop.~7.2]{Q} on the maximal spectral type of a constant-length
substitution; see also \cite[Thm.~3.4]{Bart}. The crucial point to
observe here is that we do not need to consider the spectral measures
of all functions that are square-integrable over the hull, but only
those of the (possibly weighted) lookup functions for the type of
level-$m$ supertile at $0$, for all $m\in\NN_0$.

Our diffraction measure provides the result for the spectral measure
of the lookup functions of the prototiles themselves, compare
\cite{BLvE}, while we can repeat our analysis for any supertile in
noting that this will simply lead to a rescaling, as a result of
Lemma~\ref{lem:linear-FT}. Concretely, the spectral measures will then
be Riesz products of the same type in the sense that only finitely
many initial factors are missing. Since they clearly have the same
Lyapunov exponents and growth rates, our result translates to their
spectral measures as well.  In line with \cite{CG}, but by a
completely different method, we have thus arrived at the following
result.

\begin{coro}\label{coro:spec-singular}
  Consider the dynamical system\/ $(\XX_{\varrho}, \ZZ)$ defined by
  the primitive constant-length substitution\/ $\varrho$ from
  \eqref{eq:def-gegen}, which has inflation multiplier\/ $2$. Although
  its substitution matrix also has an eigenvalue\/ $\sqrt{2}$, and
  thus satisfies the necessary criterion for the presence of an
  absolutely continuous spectral measure, no such measure exists, and
  all spectral measures are singular.  \qed
\end{coro}

It follows from the full analysis in \cite{CG} that the extremal
spectral measures are either pure point or singular continuous, and
that both possibilities occur here. Let us briefly mention that
\cite{BGM} presents a method to construct infinitely many other
examples of this kind, which demonstrates that the
$\sqrt{\lambda}\,$-criterion alone is far from sufficient for the
emergence of \textsf{ac} spectral components.

\section{Results in higher dimensions}\label{sec:general}

One advantage of the geometric language with tilings is its
generalisability to higher dimensions. Here, a \emph{tile} in $\RR^d$
is a compact set $\ct$ that is the closure of its interior, and we
will only consider cases where $\ct$ is simply connected, though this
is not required for the general theory. A \emph{prototile} is a
representative of a tile and all its translates under the action of
$\RR^d$.

Given a finite set
$\cT = \{ \ct^{\pa}_{1}, \ldots , \ct^{\pa}_{\nts L} \}$ of $L$
prototiles and an expansive linear map $Q$, one speaks of a
\emph{stone inflation} relative to $Q$ (otherwise often called a
self-affine inflation) if there is a rule how to exactly subdivide
each level-$1$ supertile $Q (\ct_i)$ into translated copies of the
original tiles.  Iterating such an inflation rule, called $\varrho$ as
before, leads to tilings that cover $\RR^d$, and via the orbit closure
in the standard local rubber topology also to a compact hull $\YY$. If
the inflation is primitive, see \cite{TAO,BL-msds,FreRi} for details,
this hull is minimal and consists of a single LI class, which is to
say that any two elements of the hull are LI. It is an interesting and
important fact that this property is not restricted to the FLC
situation, but still holds for more general inflation tilings
\cite{FR}, with the properly adjusted notions of indistinguishability
and repetitivity; see also \cite{LS}.

To keep track of the relative positions of the tiles under the
inflation procedure, we need to equip each $\ct_i$ with a reference or
control point. While there are usually many ways to do so, some will
be more `natural' than others. What really counts is that the tiling
and the point set contain the same information. So, it is imperative
to choose the control points such that they are MLD with the
tiling. When congruent tiles exist, the control points are
\emph{coloured} to distinguish them according to the tile type.  This
means that the space of (coloured) control point sets and the tiling
hull are topologically conjugate as dynamical systems under the
translation action of $\RR^d$ in a \emph{local} way. For this reason,
we usually identify the two pictures, and speak of tilings or point
sets interchangeably, always using $\YY$ to denote the hull.

Now, we can define the displacement sets $T_{ij}$ essentially as 
before, so
\begin{equation}\label{eq:T-def-gen}
    T_{ij} \, = \, \{ \text{all relative positions of } \ct_i \text{ in }
    Q (\ct_j )\} \pts ,
\end{equation}
where the relative positions are defined via the control points. Note
that all quantities are defined in complete analogy to the
one-dimensional case. In particular, the corresponding \emph{Fourier
  matrix} is once again given by
\begin{equation}\label{eq:B-def-gen}
     B (k) \, = \, \overline{ \widehat{ \delta^{\pa}_{T}}} (k) \pts .
\end{equation}
Note that $k \in \RR^d$ reflects the dimension of the Euclidean space
the tiling lives in, while $B (k) \in \Mat (L, \CC)$ covers the
combinatorial structure of the inflation rule. As before,
\[ 
     M \, = \, B(0) 
\]
is the inflation or \emph{incidence} matrix, with leading eigenvalue
$\lambda = \det (Q)$ by construction of the stone inflation. Many
explicit examples are discussed in \cite[Ch.~6]{TAO} as well as in
\cite{Natalie,Nat-review,Dirk}; see also the \textsc{Tilings
  Encyclopedia}.\footnote{The \textsc{Tilings Encyclopedia} is
  maintained by Dirk Frettl\"{o}h and Franz G\"{a}hler, and is
  accessible online at \texttt{http://tilings.math.uni-bielefeld.de}.}
Quite frequently, $Q$ will be a homothety, simply meaning
$Q (x) = \lambda x$ and thus referring to the case of a self-similar
inflation, but it can also contain a rotation (as in G\"{a}hler's
shield tiling; see \cite[Sec.~6.3.2]{TAO}) or scale differently in
different directions (as in general block substitutions; see
Figure~\ref{fig:block} below for an example).  The crucial point here
is that space and combinatorial information are properly separated for
the renormalisation approach.

The renormalisation equations for the pair correlation coefficients
are derived \cite{BGM} by the same arguments used in
Lemma~\ref{lem:ren-eq} above, where the local recognisability in the
aperiodic case follows from \cite{Boris98}.  The result reads
\begin{equation}\label{eq:renogen}
    \nu^{\pa}_{ij} (z) \, = \, \myfrac{1}{\lvert\det(Q)\rvert} 
    \sum_{m,n=1}^{L}
    \,\sum_{r\in T_{im}}  \,\sum_{s\in T_{jn}}
    \nu^{\pa}_{mn} \bigl( Q^{-1} (z+r-s) \bigr),
\end{equation}
which, in terms of the corresponding 
pair correlation measures,  becomes
\[
     \vU \, = \, \myfrac{1}{\lvert\det(Q)\rvert}  \bigl( \pts
     \widetilde{\nts \delta^{\pa}_{T} \nts }
     \overset{*}{\otimes} \delta^{\pa}_{T} \bigr) * 
     ( Q  .\vU ) \pts .
\]
Note that these relations also apply to \emph{periodic} inflation
tilings, as shown in \cite{BGM}.  Taking Fourier transforms, with the
dual map $Q^* := (Q^{T})^{-1}$, we obtain the relations
\begin{equation}
     \widehat{\vU} \, = \, \myfrac{1}{\lvert\det(Q)\rvert^2}
     \left( B (.) \otimes \overline{B (.)} \, \right)
     \bigl( Q^* \! . \pts \widehat{\vU} \pts \bigr) 
\end{equation}
by Lemma~\ref{lem:linear-FT}.  Once again, they have to hold
separately for the pure point, singular continuous and absolutely
continuous components, respectively, as in Lemma~\ref{lem:separate}.

Due to the appearance of $Q^*$, one defines the
Fourier matrix cocycle as
\begin{equation}\label{eq:f-cocycle}
   B^{(n)} (k) \, = \, B(k) \pts B(Q^T k) \pts \cdots 
   B ((Q^T )^{n-1}k) \pts,
\end{equation}
where the transpose can also be seen as a consequence of
Eqs.~\eqref{eq:T-def-gen} and \eqref{eq:B-def-gen} via a simple
calculation with the Fourier transform. Now, as in the one-dimensional
case, one has the following result \cite{BGM}.

\begin{fact}
  If\/ $B(k)$ is the Fourier matrix of the primitive stone inflation
  rule\/ $\varrho$, the Fourier matrix of\/ $\varrho^n$ is given by\/
  $B^{(n)} (k)$ from Eq.~\eqref{eq:f-cocycle}. \qed
\end{fact}

With $\chi^{B}$ as defined in Eq.~\eqref{eq:def-L}, and in complete
analogy to the one-dimensional case, one can now derive
\cite[Thm.~5.7]{BGM} the following criterion for the absence of
\textsf{ac} diffraction components.

\begin{theorem}\label{thm:higher}
  Consider a finite set\/ $\cT$ of prototiles in\/ $\RR^d$ and a
  primitive stone inflation for\/ $\cT$, with expansive linear map\/
  $Q$, and suppose that this defines an FLC tiling system.  Assume
  further that each prototile is equipped with a control point,
  possibly coloured, such that the tilings and the corresponding
  control point sets are MLD. Define Fourier matrix and Lyapunov
  exponents as explained above.
  
  Suppose that\/ $B(k)$ is invertible for a.e.\ $k\in\RR^d$ and that
  there is some\/ $\varepsilon > 0$ such that\/
  $\chi^{B} (k) \leqslant \frac{1}{2} \log \pts \lvert \det (Q)\rvert
  - \varepsilon$
  holds for a.e.\ $k\in \RR^d$. Then, one has\/
  $\widehat{\vU}_{\mathsf{ac}} =0$ and the diffraction measure of the
  tiling system is singular.  \qed
\end{theorem}

\begin{remark}
  A closer inspection of the proof in \cite{BGM} reveals that the FLC
  condition is actually not necessary. Indeed, if one starts form a
  stone inflation with finitely many prototiles up to translations,
  the criterion from Theorem~\ref{thm:higher} works without further
  modifications.  We shall see several examples later on.  \exend
\end{remark}

Let us note in passing that the comments of Remark~\ref{rem:more},
with the obvious adjustments, apply to this higher-dimensional case as
well. In particular, the conditions for the appearance of \textsf{ac}
spectral components in higher dimensions are as restrictive as in one
dimension.

\section{Binary block substitutions of constant 
size}\label{sec:constant}

An interesting class is provided by primitive, binary block
substitutions in $d$ dimensions, where we have two types of unit
blocks, white ($0$) and black ($1$) say, which are both substituted
into a block of equal size and shape. Here, we assume the
corresponding linear expansion to be
$Q = \diag (n^{\pa}_{1}, \ldots , n^{\pa}_{d} )$ with all
$n_i \geqslant 2$, so $Q = Q^T$ in this case.

Let us now place the inflated white block on top of the inflated black
one (in $\RR^{d+1}$ that is), so that one can easily identify
bijective and coincident positions via the corresponding columns.  We
cast them into polynomials as follows. Let $p$ be the polynomial in
$z = (z^{\pa}_{1}, \ldots , z^{\pa}_{d})$ for \emph{all} positions,
which means
$p (z) = \prod_{j=1}^{d} (1 + z^{\pa}_j + \ldots + z^{n_j - 1}_j )$.
Likewise, $q$ and $r$ are the polynomials for bijective columns of
type $\left[\begin{smallmatrix} 0 \\ 1 \end{smallmatrix}\right]$ and
$\left[\begin{smallmatrix} 1 \\ 0 \end{smallmatrix}\right]$, while
$s^{\pa}_0$ and $s^{\pa}_1$ stand for the polynomials of the
coincident columns of type
$\left[\begin{smallmatrix} 0 \\ 0 \end{smallmatrix}\right]$ and
$\left[\begin{smallmatrix} 1 \\ 1 \end{smallmatrix}\right]$,
respectively. Clearly, one has $q+r+s^{\pa}_0 +s^{\pa}_1 = p$.  Then,
with $k = (k^{\pa}_{1}, \ldots , k^{\pa}_{d})\in \RR^d$, the Fourier
matrix has the form
\[
     B (k) \, = \, \begin{pmatrix}
     q (z) + s^{\pa}_0 (z) & r (z) + s^{\pa}_0 (z) \\
     r (z)  + s^{\pa}_1 (z) & q (z) + s^{\pa}_1 (z)
     \end{pmatrix}  \quad \text{with} \quad
     z_j = \ee^{2 \pi \ii k_j} .
\]
Since $\det \bigl( B (k)\bigr) = p (z) \bigl( q (z) - r (z) \bigr)$,
the matrix $B (k)$ is invertible for a.e.\ $k\in\RR^d$; see
\cite{squiral} for more on bijective block substitutions and
\cite{Natalie} for some general results.

The Fourier matrix cocycle belongs to the compact dynamical system
defined by $k \mapsto Q k$ modulo $1$ on $\TT^d$.  In this situation,
we may use Oseledec's multiplicative ergodic theorem \cite{Viana}, 
which tells us that the two Lyapunov exponents exist 
for a.e.\ $k\in\RR^d$ and add up to
\[
   \lim_{N\to\infty} \myfrac{1}{N} \sum_{\ell=0}^{N-1}
   \log \pts \bigl| \det \bigl( B (Q^{\ell} k) \bigr) \bigr|
\]
whenever this limit exists, which is true for a.e.\ $k\in\RR^d$ by
Lemma~\ref{lem:Q-sample}. The limit then is
\[
   \int_{\TT^d} \log \pts \bigl| \det 
   \bigl( B (k)\bigr) \bigr|
   \dd k \, = \, \fm (p) + \fm(q-r) \, = \,
   \fm (q-r) \pts ,
\]
because  $\fm (p) = 0$ by Fact~\ref{fact:cyclo}.

Since $(1,1)$ is a left eigenvector of $B(k)$ for all $k\in\RR^d$,
with eigenvalue $p ( z^{\pa}_{1} , \ldots, z^{\pa}_{d} )$ and the $z_j$
from above, one obtains
\[
   \myfrac{1}{N} \log \pts \big \| (1,1) B^{(N)} (k) \big \|
   \, \xrightarrow{\, N\to\infty\,} \, \fm (p) \, = \, 
   \sum_{j=1}^{d} \fm \bigl( 1 + z^{\pa}_{j} + \ldots +
   z^{n_j - 1}_{j} \bigr) \, = \, 0
\]
for a.e.\ $k\in\RR^d$. Since we thus know one exponent together
with the sum, we get that
\[
    \chi^{B} (k) \, = \, \chi^{B}_{\max} (k) \, = \, \fm (q-r)
\]
holds for a.e.\ $k\in\RR^d$. By standard estimates 
\cite{Neil,BaGriMa,BGM}, one now finds
\[
   \exp \bigl( \fm (q-r)\bigr) \, < \,
   \| q-r \|^{\pa}_{1} \, \leqslant \, \| q-r \|^{\pa}_{2}
   \, = \, \sqrt{ \det (Q) - n_{\mathrm{c}}} 
   \, \leqslant \, \sqrt{ \det (Q)} \, ,
\]
where the first step follows from Jensen's inequality; see \cite{LL}
for a suitable formulation.  Moreover, with $n_{\mathrm{c}}$ denoting
the total number of coincident columns, the equality is a result of
Parseval's identity. Together, we get
$\fm (q-r) < \log \sqrt{\det(Q)}$, which gives the required criterion
for the absence of absolutely continuous components in
$\widehat{\vU}$.

Clearly, the corresponding property also holds in higher dimensions,
and we have the following general result; see also \cite{BGM}.

\begin{theorem}
  The diffraction measure of a primitive binary block substitution of
  constant size in dimension\/ $d \geqslant 2$ is always
  singular. \qed
\end{theorem}

\begin{figure}
\begin{center}
  \includegraphics[width=0.7\textwidth]{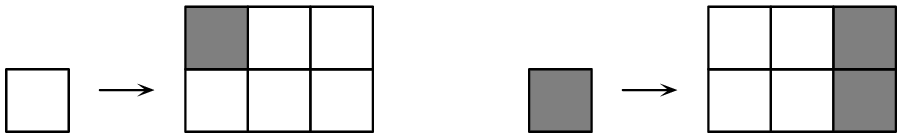}
\end{center}
\caption{A simple example of a primitive binary block substitution in
  the plane. It is a stone inflation for the linear map
  $Q = \diag (3,2)$.\label{fig:block} The lower left corners of all
  blocks are used as control points (not shown), which are coloured
  according to the block type.}
\end{figure}

\begin{example}
  For the block substitution of Figure~\ref{fig:block}, the linear
  expansion is $Q = \diag (3,2)$. With
  $(z^{\pa}_{1} , z^{\pa}_{2} ) = (x,y)$, the polynomials are
  $p (x,y) = (1+x+x^2)(1+y)$ together with $q(x,y) = x^2 (1+y)$,
  $r(x,y) = y$, $s^{\pa}_{0} (x,y) = 1 + x (1+y)$ and
  $s^{\pa}_{1} (x,y) = 0$. This gives
\[
    B (k) \, = \, \begin{pmatrix}
    1 + (x + x^2)(1 + y) & (1+x)(1+y) \\ y &
    x^2 (1+y) \end{pmatrix}
    \quad \text{with} \quad  (x,y) = \bigl(
    \ee^{2 \pi \ii k_1}, \ee^{2 \pi \ii k_2} \bigr)
\]
and $\det \bigl( B(k) \bigr) = p(x,y) ( x^2 + x^2 y - y)$. The
existence of the various limits of Birkhoff sums we need here, for
a.e.\ $k\in\RR^2$, follows once again from Lemma~\ref{lem:Q-sample}.

The non-zero Lyapunov exponent is given by
\[
\begin{split}
     \fm (x^2 + x^2 y - y) \, & = \, \fm (x^2 + x^2 y + y)
     \, = \, \fm (1 + x^{-2} y + y ) \, = \, \fm (1 + x + y ) \\
     & = \, \myfrac{3 \sqrt{3}}{4 \pts \pi} \, L (2, \chi^{\pa}_{-3} )
     \, = \, 2 \int_{0}^{1/3} \! \log \bigl( 2 \cos (\pi t) \bigr) \dd t
     \, \approx \, 0.323 {\pts} 066 \pts ,
\end{split}     
\]
where $L (z,\chi^{\pa}_{-3})$ is the $L$-function for the principal
Dirichlet character $\chi^{\pa}_{-3}$ of the imaginary quadratic field
$\QQ \bigl(\sqrt{-3}\pts\pts \bigr)$. Here, the first equality follows
from a change of variable transformation, while the third emerges by a
standard formula from \cite{EvWa}. The connection between logarithmic
Mahler measures and special values of $L$-functions is a famous result
that first appeared in \cite{Wannier} as the groundstate
entropy\footnote{It is interesting historically that this connection
  was overlooked for a long time because the numerical value given
  there (for the correct integral) was erroneous, which was corrected
  in an erratum 23 years later.}  of the anti-ferromagnetic Ising
model on the triangular lattice; see \cite{Klaus,BCM} and references
therein for more.  \exend
\end{example}

\begin{remark}
  The absence of \textsf{ac} diffraction immediately implies that the
  spectral measure for the `one-point lookup function' must be
  singular. As in the one-dimensional case, this implies that the
  spectral measure is also singular for any function that looks up the
  level-$n$ supertile at the origin; see the discussion before
  Corollary~\ref{coro:spec-singular}.  Then, by \cite{Bart},
  \emph{any} spectral measure must be singular, and our system has
  singular dynamical spectrum.  This gives another, independent proof
  of a result that was previously shown in \cite{Natalie,squiral}; see
  also \cite{Nat-review}.  \exend
\end{remark}

\section{Block substitutions with squares}\label{sec:block}

Here, we are interested in inflation rules with a single prototile of
unit area and linear expansion $Q$, but some added complexity from the
set $S$ of relative positions of the tiles within the supertile. In
particular, this will be our first class of examples where we go
beyond the FLC case. The special interest in this class originates
from the fact that the Fourier matrix cocycle, $B^{(n)} (k)$, simply
is a sequence of multivariate trigonometric polynomials. We thus write
$P^{(n)} (k)$ to indicate this. The renormalisation equation becomes
an equation directly for the autocorrelation and reads
\begin{equation}\label{eq:reno-simple}
     \gamma \, = \, \nu * (f.\gamma)
     \quad \text{with} \quad 
     \nu \, = \, \myfrac{\delta^{\pa}_{S} * 
     {\delta^{\pa}_{-S}}}{\lvert \det (Q) \rvert} \pts .
\end{equation}
Let us begin with a particularly
simple case.

\begin{example}\label{ex:square-1}
  The staggered block substitution defined by 
\[
   \raisebox{-25pt}{\includegraphics[width=0.3\textwidth]{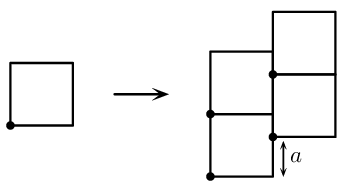}}
\]
with arbitrary $a\in\RR$ results in a tiling that is lattice-periodic,
with lattice $\vG = \langle v, e^{\pa}_{2} \rangle^{\pa}_{\ZZ}$, where
$v = e^{\pa}_{1} + a \pts e^{\pa}_{2}$. In particular, the resulting
tiling is $1$-periodic in $e^{\pa}_{2}$-direction. Since
$\dens (\vG) =1$, the autocorrelation is
$\gamma = \delta^{\pa}_{\vG}$, where we use a reference point in the
lower left corner of every unit square as indicated.  The diffraction
measure is $\widehat{\gamma} = \delta^{\pa}_{\vG^*}$ with the dual
lattice
$\vG^* = \langle e^{\pa}_{1} , e^{\pa}_{2} - a \pts e^{\pa}_{1}
\rangle^{\pa}_{\ZZ}$.

Let us look at this result from the renormalisation point of view,
which is also applicable here because the inflation can easily be
changed into a stone inflation without changing the control point
positions; see Remark~\ref{rem:stone} below. As mentioned above, the
approach also works for periodic cases, like the one at hand.  By an
iteration of the renormalisation equation \eqref{eq:reno-simple} and
an application of Lemma~\ref{lem:bei-null}, the autocorrelation is
\[
   \gamma \, = \Conv_{m\geqslant 0}
   f^m \! . \pts \nu 
   \quad \text{with} \quad
    \nu \, = \, \bigl( \delta^{\pa}_{0} + \tfrac{1}{2} (
    \delta^{\pa}_{e^{\pa}_{2}} + \delta^{\pa}_{-e^{\pa}_{2}} ) \bigr)*
   \bigl( \delta^{\pa}_{0} + \tfrac{1}{2}(\delta^{\pa}_{v} + 
   \delta^{\pa}_{-v} )\bigr) ,
\]  
  where $f (x) = 2 \pts x$. Now, Fourier transform leads to the 
  two-dimensional Riesz product
\[
    \widehat{\gamma} \, = \prod_{m\geqslant 0} \bigl( 
    1 + \cos (2 \pi \pts 2^m (k^{\pa}_{1} + a \pts k^{\pa}_{2})) \bigr)
    \bigl( 1 + \cos (2 \pi \pts 2^m k^{\pa}_{2})\bigr)
    \, = \sum_{\ell \in \ZZ} \widehat{\gamma}^{(1)}_{\ell} 
     \nts \nts \times \delta^{(2)}_{\ell} 
\]  
with
$\widehat{\gamma}^{(1)}_{\ell} \! = \prod_{m\geqslant 0} \bigl( 1 +
\cos (2 \pi \pts 2^m ( k^{\pa}_{1} + a \pts \ell )) \bigr)$
and $ k = (k^{\pa}_1 , k^{\pa}_2 )$.  Here, we have adopted the
standard notation for product distributions or measures, where the
upper index refers to the two coordinate directions.  Clearly, one has
$\widehat{\gamma}^{(1)}_{0} = \delta^{(1)}_{\ZZ}$ by
Lemma~\ref{lem:2-Riesz}.  Moreover, the distribution
$\widehat{\gamma}^{(1)}_{\ell}$ is $1$-periodic for every
$\ell\in\ZZ$, which means that $\widehat{\gamma}$ is $1$-periodic in
$e^{\pa}_{1}$-direction, independently of $a$.

Whenever $a\in\ZZ$, one finds
$\widehat{\gamma}^{(1)}_{\ell} \! = \delta^{(1)}_{\ZZ}$ by
Lemma~\ref{lem:2-Riesz}, and hence
$\widehat{\gamma} = \delta^{\pa}_{\ZZ^2}$ as required.  More
generally, given $k^{\pa}_{2}=\ell$, the only contribution to
$\widehat{\gamma}^{(1)}_{\ell}$ emerges for
$k^{\pa}_{1} + a \pts \ell = r \in \ZZ$, hence for
$k = r \pts e^{\pa}_{1} + \ell ( e^{\pa}_{2} - a \pts e^{\pa}_{1}) \in
\vG^*$,
which means that the Riesz product representation also gives
$\widehat{\gamma} = \delta^{\pa}_{\vG^*}$, as it must.

Looking at this result from the cocycle point of view gives
$P (k) = (1 + y) ( 1 + x \pts y^a) $ with $x = \ee^{2 \pi \ii k_1}$
and $y = \ee^{2 \pi \ii k_2}$, so that
\[
     \chi^{P} (k) \, :=  \lim_{n\to\infty}  \myfrac{1}{n}
     \sum_{\ell=0}^{n-1} \log \lvert P (2^{\ell} k)\rvert
     \, = \, \MM \bigl( \log \pts \lvert P\rvert \bigr) 
\] 
holds for a.e.\ $k\in\RR^2$ by Lemma~\ref{lem:gen-Sobol},
where one observes that $\log \pts \lvert P \rvert = \frac{1}{2}
\log \pts \lvert P \rvert^2$ with $\lvert P \rvert^2$ satisfying the
required conditions. Now, the mean can be calculated as
\[    
     \MM \bigl( \log \pts \lvert P\rvert \bigr) 
     \, = \, \fm^{\pa}_{y} (1+y)  \, +
     \lim_{T \to\infty} \myfrac{1}{T} \int_{0}^{T}
     \! \fm^{\pa}_{x} (1 + x \pts y^a) \dd k^{\pa}_{2} \, = \, 0 \pts ,
\]
where both Mahler measures are zero because all roots of the
polynomials lie on the unit circle. This fits with the explicit
calculation of $\widehat{\gamma}$ from above.  \exend
\end{example}

\begin{remark}\label{rem:stone}
  By standard methods, which are explained in \cite{TAO} and in
  \cite{Dirk}, one can replace the square in Example~\ref{ex:square-1}
  with a new prototile so that the inflation rule is turned into a
  stone inflation.  This has no effect on the position of the control
  point, wherefore we continue with the simpler formulation as a block
  stubstitution.

  In the same vein, one can see that our approach also works for more
  general inflation rules, certainly as long as they are MLD with a
  stone inflation. For examples of such rules and their reduction to
  stone inflations, we refer to \cite[Ch.~5]{TAO} and \cite{Dirk}.
  \exend
\end{remark}

Extending this initial example, we may consider a block substitution
with $M$ columns of $N$ blocks each, where entire columns can be
shifted in vertical direction by an arbitrary amount, as indicated in
the next diagram,
\begin{equation}\label{eq:genstag}
   \raisebox{-45pt}{\includegraphics[width=0.41\textwidth]{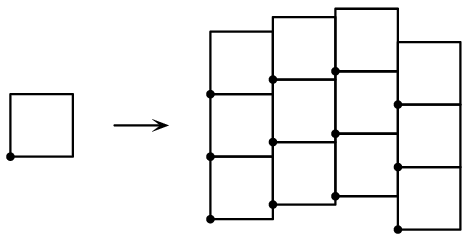}}
\end{equation}
When using the lower left corner as reference point for each square,
it is clear that a modification into a stone inflation according to
Remark~\ref{rem:stone} does not change the resulting point set. For
this reason, we stick to the formulation with squares for simplicity.

Let us now assume that the $i$th column is shifted by $a_i \in\RR$ in
vertical direction, with $i \in \{ 0,1,\ldots,M-1\}$.  Since
$a^{\pa}_{0}$ only results in a global shift of the entire block, we set
$a^{\pa}_{0} = 0$ without loss of generality, and consider the remaining
$a_i$ as shifts relative to column $0$.

As in the previous case, which had the FLC property, we define the
hull as the orbit closure of a fixed point tiling, where the closure
is now taken with respect to the \emph{local rubber topology}
\cite{BL}.  This defines a compact tiling space \cite{FR}, without any
change in the FLC case. However, this slight modification takes care
of the potential occurrence of a tiling with infinite local
complexity. As before, we obtain a dynamical system, under the
continuous translation action of $\RR^2$, which is uniquely ergodic;
see also \cite{LS}. Each tiling in the hull is $1$-periodic in the
$e^{\pa}_{2}$-direction.

From the set $S$ of relative displacements, one finds
\[
    P(k) \, = \, \overline{\widehat{\delta^{\pa}_{S}}} (k)
    \, = \, (1 + y + \ldots + y^{N-1})
    \bigl(1 + x y^{a^{\pa}_1} + x^2 y^{a^{\pa}_{2}} + \ldots +
    x^{M-1} y^{a^{\pa}_{M-1}} \bigr)
\]
with $x=\ee^{2 \pi \ii k_{1}}$ and $y=\ee^{2 \pi \ii k_{2}}$. Note
that the trigonometric polynomial $P$ is quasiperiodic, but
$1$-periodic in $k_1$. Now, with $Q = \diag (M,N)$, the cocycle is
defined by
\[
    P^{(n)} (k) \, = \, P (k) \pts P(Q k) \cdots
    P( Q^{n-1} k )  \pts ,
\]
with $P^{(1)} = P$ as usual. Our Lyapunov exponent can be
calculated as follows,
\[
    \chi^P (k) \, =  \lim_{n\to\infty} \myfrac{1}{n}
    \log \big| P^{(n)} (k) \big| \, =  \lim_{n\to\infty}
    \myfrac{1}{n} \sum_{\ell=0}^{n-1} \log \big|
    P (Q^{\ell} k) \big| .
\]
By an obvious variant of Lemma~\ref{lem:gen-Sobol}, where $\alpha$ is
replaced by the expansion $Q$, the limit exists for a.e.\ $k\in\RR^2$
and is given by
\[
\begin{split}
   \MM \bigl( \log \pts \lvert \nts P\rvert \bigr) \, & 
   = \lim_{T\to\infty}
   \myfrac{1}{T} \int_{0}^{T} \int_{0}^{1}
   \log \big| P(k) \big| \dd k^{\pa}_1 \dd k^{\pa}_2 \\
  & = \fm^{\pa}_{y} \bigl( 1 + y + \ldots + y^{N-1} \bigr) \, + 
   \lim_{T\to\infty} \myfrac{1}{T} \int_{0}^{T} \!
    \fm^{\pa}_{x} \bigl(1 + x y^{a^{\pa}_1}  + \ldots +
    x^{M-1} y^{a^{\pa}_{M-1}} \bigr)  \dd k^{\pa}_{2}  \pts .
\end{split}
\]
As $1+y+ \ldots + y^{N-1}$ is cyclotomic, the first term vanishes. The
integrand in the second is the logarithmic Mahler measure of a
polynomial (in $x$) with all coefficients on the unit circle, which is
known as a \emph{unimodular polynomial}. By an application of Jensen's
inequality in conjunction with Parseval's equation, one can show that
its logarithmic Mahler measure is bounded by $\log \sqrt{M}$ for every
$k^{\pa}_{2} \in \RR$. Consequently, we have
\[
   \MM \bigl( \log \pts \lvert \nts P\rvert \bigr) \, 
   \leqslant \, \log \sqrt{M} \, < \,
   \log \sqrt{MN}
\]
because $N>1$ by assumption. This implies that we cannot have any
absolutely continuous diffraction, and $\widehat{\gamma}$ must be
singular.

The remarkable aspect of this simple class of examples is that the
tilings are generally \emph{not} FLC; compare \cite{Dirk} and
references therein. One can say a bit more about the explicit
structure of the diffraction measure. First of all, it is $1$-periodic
in $e^{\pa}_{1}$-direction, and it consists of parallel arrangements
of one-dimensional layers, distinct in general, which have their own
Riesz product representation. We leave further details to the
interested reader.

\begin{remark}
  The above results can be generalised to $\RR^{d+1}$ with
  $d\geqslant 1$ as follows.  Consider a block of
  $M^{\pa}_{1} \times \cdots \times M^{\pa}_{d} \times N$ cubes, with
  each $M_i \geqslant 2$ and $N\geqslant 2$. Now, modify this block as
  an arrangement of $M^{\pa}_{1} \times \cdots \times M^{\pa}_{d}$
  columns of $N$ cubes each, where column
  $(m^{\pa}_{1}, \ldots, m^{\pa}_{d})$ is shifted by an arbitrary real
  number $a^{\pa}_{m^{\pa}_{1} , \ldots , m^{\pa}_{d}}$ in
  $e^{\pa}_{d+1}$-direction. We may set $a^{\pa}_{0,\ldots, 0} =0$
  without loss of generality. With
  $Q = \diag (M^{\pa}_{1}, \ldots , M^{\pa}_{d},N)$, one can now
  repeat the above analysis. Here, one obtains tilings of $\RR^{d+1}$
  that are $1$-periodic in $e^{\pa}_{d+1}$-direction.
   
  Writing
  $z = (z^{\pa}_{1}, \ldots , z^{\pa}_{d}, z^{\pa}_{d+1}) = \bigl(
  \ee^{2 \pi \ii k_1}, \ldots , \ee^{2 \pi \ii k_d}, \ee^{2 \pi \ii
    k_{d+1}}\bigr) = (x^{\pa}_{1}, \ldots , x^{\pa}_{d}, y)$,
  one finds
  \[
  \begin{split}
       P (k) \, & = \, \bigl( 1 + y + \ldots + y^{N-1} 
       \bigr) R(k)        \quad \text{with} \\[1mm]
       R (k) \, & = \sum_{m^{\pa}_{1} = 0}^{M_{1} -1} \cdots
       \sum_{m^{\pa}_{d} = 0}^{M_{d} - 1}
       x^{m_1}_{1} \nts \cdots \pts x^{m_d}_{d} \, 
       y^{a^{\pa}_{m_1, \ldots , m_d}} ,
  \end{split}
  \]
  where $R$, and hence also $P$, is $1$-periodic in $e_i$-direction 
  for all $1\leqslant i \leqslant d$. The maximal Lyapunov exponent,
  for a.e.\ $k\in\RR^{d+1}$, is now given by
\[
\begin{split}
     \chi^{B} (k) \, & = \lim_{T\to\infty} \myfrac{1}{T} \int_{0}^{T} \!
     \fm^{\pa}_{x} (R) \dd k^{\pa}_{d+1} \, \leqslant \, 
     \log \sqrt{M_{1} \cdots M_{d}\,} \\[1mm]
     & = \, \myfrac{1}{2} \sum_{i=1}^{d} \log (M_i)
     \, < \, \myfrac{1}{2} \log \bigl( \det (Q)\bigr) ,
\end{split}
\]  
 where the last estimate is a consequence of $N\geqslant 2$, while the
 intermediate steps work in complete analogy to our above treatment
 for $d=1$. The conclusion is, once again, the absence of absolutely
 continuous diffraction.
   \exend
\end{remark}

Obviously, one can extend this class of examples by colouring
blocks. This will lead to higher-dimensional Fourier matrices again,
with an uncoloured tiling of the above type as a factor system.  We
leave further details to the interested reader, and turn to a perhaps
more interesting non-FLC example.

\section{The Frank--Robinson tiling}\label{sec:FR}

Let us take a closer look at the tiling dynamical system defined by
the stone inflation\smallskip 
\begin{equation}\label{eq:FR-rule}
    \raisebox{-23pt}{\includegraphics[width=0.9\textwidth]{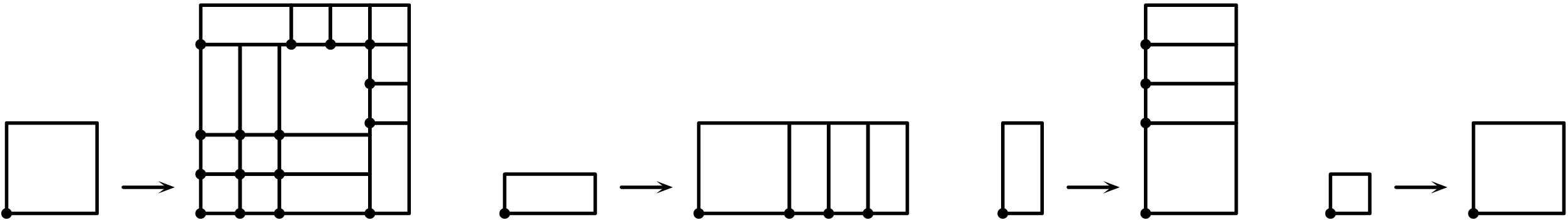}}\smallskip
\end{equation}
where the short edge has length $1$ and the long one length
$\lambda = \frac{1}{2} \bigl( 1 + \sqrt{13}\,\bigr) \approx 2.303$,
while the linear expansion is $Q = \lambda \pts \one^{\pa}_2$.  This
inflation defines the Frank--Robinson tiling \cite{FR}, see also
\cite{primer}, a patch of which is shown in Figure~\ref{fig:FR}.  Note
that the algebraic integer $\lambda$ is neither a PV number nor a
unit. By standard PF theory, the relative prototile frequencies in any
Frank--Robinson tiling are given by
\begin{equation}\label{eq:FRfreq}
   (\nu^{\pa}_{1}, \ldots , \nu^{\pa}_{4}) \, = \,
    \myfrac{1}{9} (4-\lambda, 4 \lambda -7, 
    4 \lambda -7, 19-7\lambda) \pts ;
\end{equation}
see \cite[Ex.~5.8]{TAO} for details. With the chosen edge lengths, the
density of the control point set $\vL$ induced by
Eq.~\eqref{eq:FR-rule} is $\dens (\vL) = (3+\lambda)/13 \approx
0.408$. 

\begin{figure}
\begin{center}
  \includegraphics[width=0.7\textwidth]{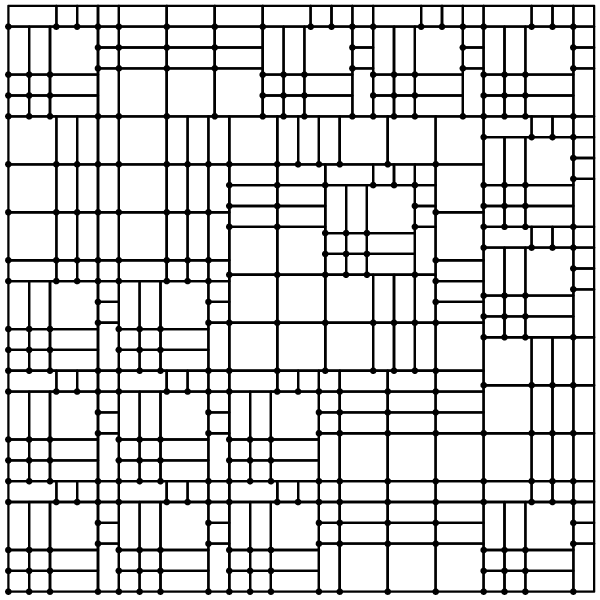}
\end{center}
\caption{A patch of the Frank--Robinson tiling\label{fig:FR} defined
  by the stone inflation rule~{\eqref{eq:FR-rule}}, obtained by three
  inflation steps from a single large square.}
\end{figure}

If we define the hull as the orbit closure of a fixed point tiling
(under the square of the rule~\eqref{eq:FR-rule}) in the local rubber
topology, we get a compact tiling space of infinite local complexity;
see \cite{FR} and \cite[Ex.~5.8]{TAO} for more. As in the FLC case, it
is also true here \cite[Cor.~5.7 and Ex.~6.3]{LS} that the tiling
dynamical system does not have any non-trivial eigenfunction.  In the
diffraction context, this implies that the \textsf{pp} part consists
of the trivial Bragg peak at $k=0$ only; see below for its intensity.

By taking the lower-left corner of each prototile as its control
point, as indicated in Eq.~\eqref{eq:FR-rule} and Figure~\ref{fig:FR},
we turn each tiling of the hull into a Delone set that is MLD with the
tiling. Now, the absence of non-trivial eigenfunctions translates to
the diffraction of this Delone set by asserting that the trivial Bragg
peak at $k=0$ is the only contribution to the pure point part of the
diffraction measure.

The Fourier matrix $B$ is given by
\begin{equation}\label{eq:B-def}
      B (x,y) \, = \, \begin{pmatrix}
      x^2 \pts y^2 & 1 & 1 & 1 \\
      p(x,y) & 0 & r(y) & 0 \\
      p(y,x) & r(x) & 0 & 0 \\
      q(x,y) & 0 & 0 & 0 \end{pmatrix} ,
      \qquad \text{with } (x,y) = \bigl(
      \ee^{2 \pi \ii k_1}, \ee^{2 \pi \ii k_2} \bigr)
\end{equation}
and the (trigonometric) polynomials
\begin{equation}\label{eq:further}
\begin{split}
    r (x) \, &  = \, x^{\lambda} + x^{\lambda + 1} + x^{\lambda + 2} , \\
   p (x,y) \, & = \,  x^2 + x^2 \pts y + y^{\lambda+2}, \\
   q (x,y) \, & = \, 1 + x + y + x \pts y + x^{\lambda} y^{\lambda} 
      \bigl( x^2  + y^2  + x \pts y^2 + x^2 y + x^2 y^2\bigr) .
\end{split}
\end{equation}
Note that $B(0)$ is the inflation matrix of the tiling, with PF
eigenvalue $\lambda^2$.

Now, the cocycle is given by
$B^{(n)} (k) = B(k) B(\lambda k) \cdots B(\lambda^{n-1} k)$, with
maximal Lyapunov exponent
\[
    \chi^{B} (k) \, = \, \limsup_{n\to\infty}
    \myfrac{1}{n} \log \| B^{(n)} (k) \| \pts ,
\]
where the choice of the (sub-multiplicative) matrix norm is
arbitrary. Absence of absolutely continuous diffraction will be
implied if we show that
$\chi^{B} (k) \leqslant \log (\lambda) - \varepsilon$ for some
$\varepsilon > 0$ and a.e.\ $ k \in\RR^2$. Since it is convenient to
work with the square of the Frobenius norm,\footnote{The spectral norm
  gives better bounds, but is harder to calculate. Also, computing
  means is easier with simple trigonometric polynomials, via
  harvesting their quasiperiodicity.} we prefer to compare
$2\pts\pts \chi^{B}$ with $\log (\lambda^2) \approx 1.668$ instead.

By standard subadditive arguments along the lines used for our
previous examples, one finds that, for any $N\in\NN$ and
then a.e.\ $k \in \RR^2$, 
\begin{equation}\label{eq:upper}
   2 \pts\pts \chi^{B} (k) \, \leqslant \,
   \myfrac{1}{N} \, \MM \bigl( \log \| B^{(N)} (.) 
   \|^{2}_{\mathrm{F}} \bigr) \, =: \, \fm^{\pa}_N \pts ,
\end{equation}
where $\MM$ again denotes the mean. Since
$ \log \| B^{(N)} (.)  \|^{2}_{\mathrm{F}}$ is a quasiperiodic
function in two variables, with fundamental frequencies $1$ and
$\lambda$, the mean can be expressed as an integral over the
$4$-torus, $\TT^4$. To this end, one introduces new variables
$u^{\pa}_{1}, u^{\pa}_{2}$ and $v^{\pa}_{1}, v^{\pa}_{2}$ such that
\[
    B (k) \, = \, \tilde{B} (u^{\pa}_{1}, u^{\pa}_{2},
    v^{\pa}_{1}, v^{\pa}_{2}) \big|_{u^{\pa}_{1} = \lambda k^{\pa}_1 \pts ,
    \, u^{\pa}_{2} = k^{\pa}_1 \pts , \, v^{\pa}_{1} = \lambda k^{\pa}_2 \pts , \,
    v^{\pa}_{2} = k^{\pa}_2 }
\]
where $\tilde{B}$ is $1$-periodic in each variable. Here, $\tilde{B}$
is defined in complete analogy to \eqref{eq:B-def}, with the
corresponding modifications on $r$, $p$ and $q$ from
Eq.~\eqref{eq:further}; compare \cite{BFGR} for a related
one-dimensional case analysed previously.  One now finds
\begin{equation}\label{eq:upperbound}
   \fm^{\pa}_{N} \, = \, \myfrac{1}{N} \, 
   \MM \bigl( \log \| \tilde{B}^{(N)} (.) 
   \|^{2}_{\mathrm{F}} \bigr) \, = \,
   \myfrac{1}{N} \int_{\TT^4} \log 
   \| \tilde{B}^{(N)} (u^{\pa}_{1}, u^{\pa}_{2}, v^{\pa}_{2},
    v^{\pa}_{2})   \|^{2}_{\mathrm{F}}
   \dd u^{\pa}_{1} \dd u^{\pa}_{2} \dd v^{\pa}_{1} \dd v^{\pa}_{2} \pts ,
\end{equation}
which can be calculated numerically with good precision.
Figure~\ref{fig:compare} illustrates the result.

\begin{figure}
\begin{center}
  \includegraphics[width=0.83\textwidth]{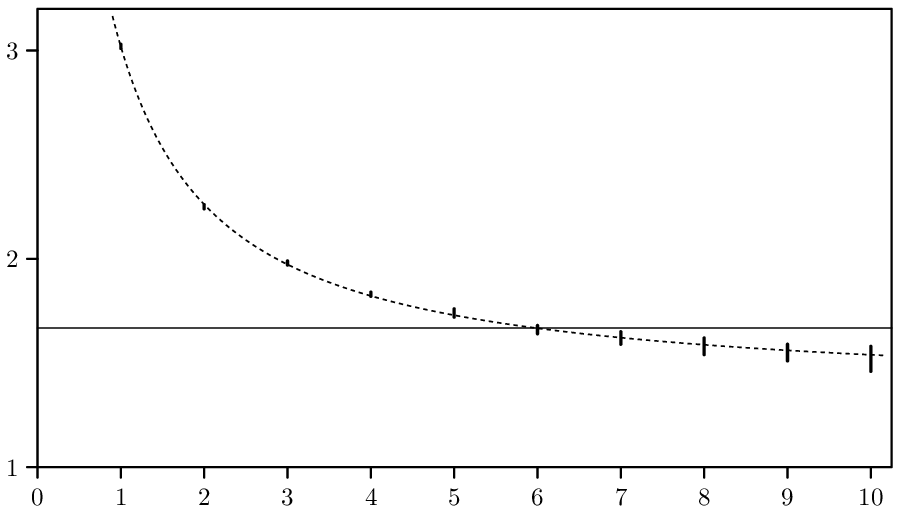}
\end{center}
\caption{Numerical values of the upper bounds $\fm^{\protect \pa}_{N}$
  of $2 \pts \chi^{B}$ from Eq.~\eqref{eq:upperbound}, for
  $1 \leqslant n \leqslant 10$. The horizontal line is at height
  $2 \log (\lambda) \approx 1.668$.\label{fig:compare} The estimated
  numerical errors are indicated by vertical bars (the dotted line is
  for eye guidance only).}
\end{figure}

Let us summarise this section as follows.  

\begin{theorem}
  Let $\vL$ be the set of control points of any element of the
  Frank--Robinson tiling hull. Then, the diffraction measure of the
  corresponding Dirac comb, $\delta^{\pa}_{\! \vL}$, is of the form\/
  $\widehat{\gamma} = \dens (\vL)^2 \, \delta^{\pa}_{0} +
  \widehat{\gamma}^{\pa}_{\mathsf{sc}}$,
  where the singular continuous part can be expressed in terms of a
  generalised Riesz product.
  
  More generally, if one assigns general complex weights\/
  $u^{\pa}_{1}, \ldots, u^{\pa}_{4}$ to the four types of control points,
  not all zero, the corresponding diffraction measure is still
  singular, where the only point measure is the central peak at\/ $0$
  with intensity
\[
  I_0 \, = \, \dens (\vL)^2 \, \big| \nu^{\pa}_{1} u^{\pa}_{1} + 
   \ldots + \nu^{\pa}_{4} u^{\pa}_{4} \big|^2 ,
\]
where\/ $\dens (\vL) = (3+\lambda)/13$ and the\/ $\nu^{\pa}_{i}$ are the
prototile frequencies from Eq.~\eqref{eq:FRfreq}. \qed
\end{theorem}

The natural next step consists in defining the (integrated)
\emph{distribution function} for $\widehat{\gamma}^{\pa}_{\mathsf{sc}}$
in the positive quadrant, as
\[
     F (k^{\pa}_{1}, k^{\pa}_{2}) \, = \,
      \widehat{\gamma}^{\pa}_{\mathsf{sc}}
     \bigl( [0,k^{\pa}_{1}]\! \times\! [0,k^{\pa}_{2}] \bigr),
\]
with the matching extension to the other quadrants. This leads to a
continuous function (which requires an extra argument along the
directions of $e^{\pa}_{1}$ and $e^{\pa}_{2}$; compare
\cite[Sec.~5]{squiral} for a similar analysis) which behaves as $F
(k^{\pa}_{1}, k^{\pa}_{2}) \sim \gamma(\{0\})\, k^{\pa}_{1}
k^{\pa}_{2}$ for large values of $k^{\pa}_{1}$ and $k^{\pa}_{2}$.  As
such, it does not reveal the interesting structure of the \textsf{sc}
measure. A better understanding of the latter requires a multi-fractal
analysis, which is outside the scope of this survey.

At this point, it is suggestive to assume that also the dynamical
spectrum is singular, in particular in the light of \cite{BLvE}, but
we have no complete answer to this question at present.

\section{Closing remarks}\label{sec:outlook}

As we have illustrated by various examples, Lyapunov exponents lead to
useful insight on the spectral nature of inflation tilings in any
dimension. They are a powerful tool to exclude absolutely continuous
spectral components. So far, our approach is taylored to inflation
tiling spaces with finitely many prototiles up to translations, and
thus gives no new insight to pinwheel-type systems. Nevertheless, the
latter also have a strong renormalisation structure, and further
progress seems possible.

To explore the absence of \textsf{ac} diffraction and spectral
measures in more generality, one would need a more analytic (rather
than numerical) approach to the estimates for upper bounds, or,
ideally, exact expressions of F\"{u}rstenberg type for the
exponents. Also, it would help to establish the almost sure existence
of Lyapunov exponents as limits, which does not seem to be an easy
task outside the constant-length or the Pisot case.

In this exposition, we have mainly considered the absolutely
continuous part of the spectrum. It is not difficult to analyse the
pure point part as well, where some results are discussed in
\cite{BG15,BGM}. Considerably more involved seems the singular
continuous part, which originates from the different scalings one
encounters. Various results on the spectral measures in one dimension
are derived in \cite{BuSol,BuSol2} via matrix Riesz products, which
are not restricted to the self-similar case.  It would be interesting
to establish a connection with the topological constraints on size and
shape changes \cite{CS,CS2}, which should at least be possible in the
irreducible Pisot case.

\section*{Acknowledgements}

It is our pleasure to thank Alan Bartlett, Michael Coons, Natalie
Frank, Franz G\"{a}hler, Alan Haynes, Neil Ma\~{n}ibo, Robbie
Robinson, Dan Rust, Lorenzo Sadun and Boris Solomyak for discussions
and helpful comments on the manuscript.  Various useful hints from an
anonymous reviewer are gratefully acknowledged. This work was
supported by the German Research Foundation (DFG), within the CRC~1283
at Bielefeld University, and by EPSRC through grant EP/S010335/1.

\end{document}